\documentclass[reqno]{amsart}
\usepackage[utf8]{inputenc}
\usepackage{amsthm,amsfonts,amstext,amssymb,mathrsfs,amsmath,latexsym}
\usepackage{cite} %for citation in numeric order
\usepackage{enumerate}
\usepackage[abs]{overpic}
\usepackage{epsfig}
\usepackage{graphicx}
\usepackage{hyperref}
\usepackage{comment}

\usepackage{tikz}

\usepackage{caption}
\usepackage{subcaption}

\newtheorem{thm}{Theorem}[section]
\newtheorem{prop}[thm]{Proposition}
\newtheorem{lem}[thm]{Lemma}
\newtheorem{cor}[thm]{Corollary}

\theoremstyle{definition}
\newtheorem{definition}[thm]{Definition}

\theoremstyle{remark}

\numberwithin{equation}{section}

\DeclareMathOperator{\dist}{dist} % The distance.
\DeclareMathOperator{\re}{Re}
\DeclareMathOperator{\im}{Im}
\DeclareMathOperator{\supp}{supp}

\DeclareMathOperator{\capac}{cap}
\DeclareMathOperator{\disth}{d_H}
\DeclareMathOperator{\diam}{diam}

\DeclareMathOperator{\distp}{d}

\newcommand{\C}{\mathbb{C}}
\newcommand{\R}{\mathbb{R}}

\newcommand{\Boh}{\mathcal{O}}
\newcommand{\boh}{\mathit{o}}
\newcommand\restr[2]{{% we make the whole thing an ordinary symbol
  \left.\kern-\nulldelimiterspace % automatically resize the bar with \right
  #1 % the function
  \vphantom{\big|} % pretend it's a little taller at normal size
  \right|_{#2} % this is the delimiter
  }}

\begin{document}

%%
%% The title of the paper goes here.  Edit to your title.
%%

\title{S-curves in  polynomial external fields}

\author{Arno B. J. Kuijlaars and Guilherme L. F. Silva}
\address{KU Leuven, Department of Mathematics, Celestijnenlaan 200B bus 2400, B-3001 Leuven, Belgium}
\email{arno.kuijlaars@wis.kuleuven.be, guilherme.silva@wis.kuleuven.be}
\thanks{The first author is supported by KU Leuven Research Grant OT/12/073, 
the Belgian Interuniversity Attraction Pole P07/18, FWO Flanders projects
G.0641.11 and G.0934.13, and by Grant No. MTM2011-28952-C02 of the Spanish 
Ministry of Science and Innovation. The second author is supported by 
FWO Flanders projects G.0641.11 and G.0934.13.}

\dedicatory{Dedicated to the memory of Andrei Aleksandrovich Gonchar and Herbert Stahl}

%\urladdr{www.}

%\author{Second Author}
%\address{}
%\email{}
%\urladdr{}

%%%
%%% The following is for the abstract.  The abstract is optional and
%%% if not used just delete, or comment out, the following.
%%%

\begin{abstract}

Curves in the complex plane that satisfy the S-property were first introduced
by Stahl and they were further studied by  Gonchar and Rakhmanov in the 1980s.
Rakhmanov recently showed the existence of curves with the S-property in a harmonic
external field by means of a max-min variational problem in logarithmic potential theory.
This is done in a fairly general setting, which however does not include the important special
case of an external field $\re V$ where $V$ is a polynomial  of degree $\geq 2$.
In this paper we give a detailed proof of the existence of a curve with the S-property
in the external field $\re V$ within the collection of all curves that connect two 
or more pre-assigned directions  at infinity in which $\re V \to +\infty$.  
Our method of proof is very much based on the works of Rakhmanov on the max-min variational 
problem and of Mart\'inez-Finkelshtein and Rakhmanov on critical measures. 
 
\end{abstract}

\maketitle

\tableofcontents

%%%%%%%%%%%%%%%%%%%%%%%%%%%%%%%%%%%%%%%%%%%%%%%%%%%%%%%%%%%%%%%%%%%%%%
\section{Introduction}
%%%%%%%%%%%%%%%%%%%%%%%%%%%%%%%%%%%%%%%%%%%%%%%%%%%%%%%%%%%%%%%%%%%%%%

The concept of {\it S-property} for logarithmic potentials appeared for the first time in the works of H. Stahl
\cite{stahl_extremal_domains,stahl_orthogonal_polynomials_complex_measures,stahl_orthogonal_polynomials_complex_weight_function,
stahl_structure_extremal_domains}, in the study of the limiting distribution of poles of Pad\'e approximants, 
see also \cite{stahl_sets_minimal_capacity} for a recent survey and extension of these results. 

Inspired by Stahl's works, Gonchar and Rakhmanov extended the S-property to situations with 
external field \cite{gonchar_rakhmanov_rato_rational_approximation}. They then characterized the 
limiting distribution of zeros of certain {\it non-hermitian} orthogonal polynomials, subject to the existence of a 
certain curve with the S-property in an external field - the {\it S-curve}, see Section \ref{section_s_property} 
below for a precise definition.

Very recently, Rakhmanov returned to the question of existence of S-curves from a general 
perspective \cite{rakhmanov_orthogonal_s_curves}. He considered an associated {\it max-min equilibrium problem}
in logarithmic potential theory, and in an ``ideal'' situation he proved that this max-min problem 
has a solution and the solution  has the S-property. This general approach by Rakhmanov is very similar 
in spirit to his previous work with Kamvissis \cite{kamvissis_rakhmanov_energy_maximization}, where a max-min
problem for Green energy in a particular external field was considered.

In \cite{rakhmanov_orthogonal_s_curves}, Rakhmanov also pointed out some examples where his proposed 
``ideal'' situation does not apply, although similar considerations should also
lead to the existence of a curve with the S-property. The present work is aimed at studying this problem
in a very particular, but important, model, namely when the external field is given
by the real part of a polynomial.

Following the approach of Rakhmanov, for a polynomial $V$ we consider a max-min problem of the form
$$
\max_{\Gamma}\min_{\substack{\supp\mu\subset \Gamma \\ \mu(\C)=1}}
	\left[ \iint\log\frac{1}{|x-y|}d\mu(x)d\mu(y)+\int \re V(x)\, d\mu(x) \right],
$$
where the maximum is taken over a suitable class of contours $\Gamma$ (see Section \ref{section_admissible_contours} 
below for precise definitions), and the infimum is taken over the set of Borel probability measures $\mu$ on $\Gamma$. 
The aim of this work is to prove that the max-min problem has a solution, and this solution leads to an S-curve $\Gamma$.

The max-min approach is not the only possible approach to the existence of the S-curve. In some particular cases, 
it is possible to construct the S-curve explicitly from specific properties of the problem at hand, see e.g.
\cite{alvarez_alonso_medina_s_curves,aptekarev_arvesu_meixner, atia_martinez_martinez_thabet, claeys_wielonsky, 
bertola_tovbis_quartic_weight, deano_huybrechs_kuijlaars} for recent contributions. 
In general, the explicit determination of the S-curve - 
or more specifically the determination of the support of its equilibrium measure - is a hard problem, and some
numerical studies have also been carried out, see \cite{alvarez_alonso_medina_s_curves,bertola_tovbis_quartic_weight}.

Motivated by questions in Random Matrix Theory, Bertola \cite{bertola_boutroux} also studied the existence of the 
S-curve for polynomial external field in the same framework as the present paper. 
By using deformation techniques, Bertola was able to show the existence of the S-curve. 
However, his approach just works when the underlying equilibrium
measure vanishes as a square root on the endpoints of its support, although it is very likely that his 
approach could also be adapted to remove this restriction. It is worth
emphasizing that in our work there is no restriction on the underlying equilibrium measure -
so we get the same result as Bertola but valid in the general situation.

An extension of the concept of S-curves to vector equilibrium problems appears in
the context of Hermite-Pad\'e approximation, see e.g.\ \cite{aptekarev_kuijlaars_vanassche,nuttall,rakhmanov_suetin}.
Vector equilibrium problems also appear in certain problems in random matrix theory,
for example for random matrix models with external source 
\cite{aptekarev_bleher_kuijlaars,bleher_delvaux_kuijlaars} and coupled random matrices \cite{duits_kuijlaars_mo},
see \cite{kuijlaars_survey} for an overview. The asymptotic analysis has been mostly restricted
to cases with enough symmetry so that the relevant contours are contained
in the real or imaginary axis. 
 A suitable existence theory for curves with S-property appropriate to vector
equilibrium problems with external fields, and possibly also upper constraints, would
be a major step towards the analysis of situations with less symmetry.
See \cite{aptekarev_lysov_tulyakov} for a first example in this direction.

%%%%%%%%%%%%%%%%%%%%%%%%%%%%%%%%%%%%%%%%%%%%%%%%%%%%%%%%%%%%%%%%%%%%%%
\section{Background and statement of the main result}\label{previous_remarks}
%%%%%%%%%%%%%%%%%%%%%%%%%%%%%%%%%%%%%%%%%%%%%%%%%%%%%%%%%%%%%%%%%%%%%%

To state our main result, we first need to introduce a few notations and notions.

Throughout this paper, $V$ always denotes a complex polynomial of degree $N\geq 2$ and 
$$
\varphi=\re V.
$$
Also, we denote
$$
D_R(z_0)=\{z\in \C \mid |z-z_0|<R\}, \qquad D_R=D_R(0).
$$

\subsection{Notions from potential theory} \label{potential_theory_notions}

Given a finite Borel measure $\mu$ on $\C$, its {\it logarithmic potential} at $x\in \C$ is defined by
$$
U^\mu(x)=\int \log\frac{1}{|x-y|} \, d\mu(y),
$$
and its {\it logarithmic energy} by
$$
I(\mu)=\iint \log\frac{1}{|x-y|} \, d\mu(x)d\mu(y)=\int U^\mu(x) \, d\mu(x),
$$
whenever these integrals make sense.

For a closed set $F\subset \C$ and $\varphi$ the real part of a polynomial as before, 
we define $\mathcal{M}_1(F)$ as the set of Borel probability measures $\mu$ supported on $F$, satisfying
the growth condition 
\begin{equation}\label{integrability_at_infinity}
\int\log(1+|x|)d\mu(x)<+\infty,
\end{equation}
and by $\mathcal{M}^\varphi_1(F)$ we denote the subset of $\mathcal M_1(F)$ consisting 
of all measures $\mu$ for which $I(\mu)$ is finite and $\varphi\in L^1(\mu)$.
The condition \eqref{integrability_at_infinity} assures the quantities $U^\mu(z),I(\mu)$ are well defined elements of
$(-\infty,+\infty]$ and $U^\mu(z)$ is finite for a.e.\ $z\in\C$ with respect to planar Lebesgue measure.

For $\mu\in \mathcal{M}^\varphi_1(F)$, the quantity
$$
I^\varphi(\mu)=I(\mu)+\int \varphi(x)d\mu(x),
$$
is called the {\it weighted logarithmic energy of $\mu$ in the external field $\varphi$}, 
or just weighted energy of $\mu$.

The following minimization problem is classical \cite{Saff_book}: find a measure 
$\mu^\varphi=\mu^{\varphi,F} \in \mathcal{M}^\varphi_1(F)$ for
which the infimum
\begin{equation}\label{definition_energy_functional}
I^\varphi(F)=\inf_{\mu \in \mathcal{M}^\varphi_1(F)} I^\varphi(\mu)
\end{equation}
is attained, i.e., for which $I^\varphi(\mu^\varphi)=I^\varphi(F)$. The measure 
$\mu^\varphi$ is called the {\it equilibrium measure} of
the set $F$ in the external field $\varphi$. For the case without external field, 
that is, when $\varphi\equiv 0$, the measure $\mu^0$ is also called the {\it Robin
measure} of $F$, the value $I(F)$ is the {\it Robin constant} of $F$, the constant $e^{-I(F)}$ 
is called {\it (logarithmic) capacity} of $F$ and denoted by $\capac F$.

For $\varphi$ the real part of a polynomial as before, if $\Gamma$ is a contour in $\C$ satisfying the growth condition
\begin{equation}\label{growth_condition}
 \lim_{\substack{z\to\infty \\ z\in\Gamma}}\varphi(z)=+\infty,
\end{equation}
the equilibrium measure $\mu^{\varphi,\Gamma}$ uniquely exists \cite{Saff_book} and 
can be characterized through the {\it Euler - Lagrange variational conditions}. It is
the only measure in $\mathcal M_1^\varphi(\Gamma)$ for which there exists a constant $l$ satisfying
\begin{align}
 U^{\mu}(z)+\frac{1}{2}\varphi(z) & = l,  \quad z\in\supp\mu, \label{variational_condition_1}\\
 U^{\mu}(z)+\frac{1}{2}\varphi(z) & \geq l, \quad z\in \Gamma. \label{variational_condition_2}
\end{align}

\subsection{S-property}\label{section_s_property}

\begin{definition}\label{definition_s_property}
A contour $\Gamma$ satisfying the growth condition \eqref{growth_condition} is said to possess the 
{\it S-property in the external field} $\varphi$ if its equilibrium measure
$\mu=\mu^{\varphi,\Gamma}$ in the external field satisfies the following. There is a finite set $E$ such that  
$\supp\mu\setminus E$ is locally an analytic arc at each of its points, and its 
potential $U^\mu = U^{\mu^{\varphi,\Gamma}}$ satisfies
\begin{equation} \label{Sproperty}
 \frac{\partial}{\partial n_+}\left( U^{\mu}+\frac{1}{2}\varphi\right)(z)= 
	\frac{\partial}{\partial n_-}\left( U^{\mu}+\frac{1}{2}\varphi\right)(z), \quad
	z\in\supp\mu\setminus E,
\end{equation} 
where $n_\pm$ are the unit normal vectors to $\supp\mu$ at $z$.
\end{definition}
If a contour has the S-property in the external field $\varphi$, we also call it an {\it S-curve} in the 
external field $\varphi$.
If the external field $\varphi$ is understood (as it is in this paper), we briefly call it an S-curve.

As it is clear already from the Definition \ref{definition_s_property}, the S-property is intrinsic to 
the equilibrium measure in the external field of the contour. If we modify the contour away from the support, 
keeping the equilibrium measure in the external field the
same for the modified contour, the S-property is preserved.

Although in the present work we are just interested in the S-property stated in terms of
logarithmic potentials, its analogue for potentials coming from other kernels, such as
the Green potential, have also their own interest \cite{kamvissis_rakhmanov_energy_maximization,
martinez_rakhmanov_suetin_variation_equilibrium_energy,rakhmanov_suetin}.

\subsection{The class of admissible contours}\label{section_admissible_contours}

There exist $N$ sectors $S_1,\hdots,S_N$ in the complex plane along which 
$\varphi(z) = \re V(z) \to +\infty$ as $z\to \infty$ in $S_j$. These sectors are determined by the leading coefficient
of $V$. If $V(z) = a_0 z^N + \cdots$, $a_0 \neq 0$, then 
\begin{equation}\label{definition_admissible_angles}
	S_j=\left\{z\in\C \mid |\arg z - \theta_j|<\frac{\pi}{2N}\right\}, \qquad \theta_j=-\frac{\arg a_0}{N}+\frac{2\pi(j-1)}{N},
\end{equation}
for $j=1,\hdots,N$.

We say that a set $F\subset \C$ {\it stretches out to infinity} in the sector $S_j$ if 
there exist $\epsilon>0$ and $r_0>0$ such that for every $r>r_0$ there is $z\in F$ with
$$
|z|=r,\qquad |\arg z-\theta_j|<\frac{\pi}{2N}-\epsilon.
$$
Thus in particular we have $\varphi(z) \to +\infty$ if $ z \to \infty$ within $S_j\cap F$.

A partition $\mathcal P$ of $\{1,\hdots,N\}$ is called {\it noncrossing} if it satisfies the following property:
$$
j \stackrel{\mathcal P}{\sim}j' \ \wedge \ k \stackrel{\mathcal P}{\sim} k' \ \wedge j<k<j'<k' \Longrightarrow 
j \stackrel{\mathcal P}{\sim}k.
$$
Here $\stackrel{\mathcal P}{\sim}$ denotes the equivalence relation that is associated with the partition $\mathcal P$.

The term {\it noncrossing} becomes more evident in a graphical representation of a partition, 
see Figure \ref{examples_partitions_crossing_noncrossing}.

% \begin{center}
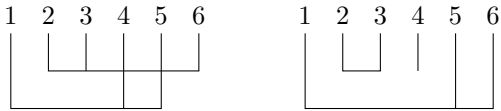
\begin{figure}[h]
\begin{subfigure}{.3\textwidth}
\centering
\begin{tikzpicture}[scale=.5]
 \draw (0,0) -- (0,-2) -- (3,-2) -- (3,0);
 \draw (3,-2) -- (4,-2) -- (4,0);
 \draw (1,0) -- (1,-1) -- (5,-1) -- (5,0);
 \draw (2,0) -- (2,-1);
 \node[above] at (0,0) {1};
 \node[above] at (1,0) {2};
 \node[above] at (2,0) {3};
 \node[above] at (3,0) {4};
 \node[above] at (4,0) {5};
 \node[above] at (5,0) {6};
\end{tikzpicture}
\end{subfigure}
\begin{subfigure}{.3\textwidth}
\centering
\begin{tikzpicture}[scale=.5]
 \draw (0,0) -- (0,-2) -- (5,-2) -- (5,0);
 \draw (4,0) -- (4,-2);
 \draw (1,0) -- (1,-1) -- (2,-1) -- (2,0);
 \draw (3,0) -- (3,-1);
 \node[above] at (0,0) {1};
 \node[above] at (1,0) {2};
 \node[above] at (2,0) {3};
 \node[above] at (3,0) {4};
 \node[above] at (4,0) {5};
 \node[above] at (5,0) {6};
\end{tikzpicture}
\end{subfigure}
\caption{Graphical representation of the crossing partition $\mathcal P = \{\{ 1,4,5 \},\{2,3,6\}\}$ (on the left) 
and of the noncrossing partition
$\mathcal P = \{\{1,5,6\},\{2,3\},\{4\}\}$ (on the right).}\label{examples_partitions_crossing_noncrossing}
\end{figure}
%\end{center}

To any noncrossing partition we associate a collection $\mathcal T$ of admissible contours.

\begin{definition}\label{definition_admissible_sets}
Let $\mathcal P$ be a noncrossing partition of $\{1,\hdots,N\}$ and $\mathcal P_0$ 
the subset of $\mathcal P$ obtained  by removing from $\mathcal P$ all singleton sets. We assume $\mathcal P_0 \neq \emptyset$.
We associate with $\mathcal P$ the collection $\mathcal T(\mathcal P)$ of admissible contours $\Gamma$ defined as follows.
\begin{itemize}
 \item[ i)] Each $\Gamma\in \mathcal T(\mathcal P)$ is a finite union of $C^1$ Jordan arcs.
 
 \item[ ii)] Each $\Gamma\in \mathcal T(\mathcal P)$ has at most $|\mathcal P_0|$ connected components, 
all of them unbounded and stretching out to infinity in at least two of the sectors $S_1,\hdots,S_N$.
 
 \item[ iii)] For each $A\in\mathcal P_0$ there is a connected component $\Gamma_A$ of $\Gamma$ that stretches 
out to infinity in each sector $S_j$ with $j\in A$.

\item[iv)] If $A\in\mathcal P\setminus \mathcal P_0$ then there exists $R>0$ for which
$$
\Gamma\cap \left(S_j\setminus D_R\right)=\emptyset, \quad j\in A.
$$
\end{itemize}

For simplicity, we will mostly write $\mathcal T$ instead of $\mathcal T(\mathcal P)$.
\end{definition}

% \begin{center}
\begin{figure}
  \begin{subfigure}{.3\textwidth}
  \centering
  \begin{overpic}[scale=0.2]{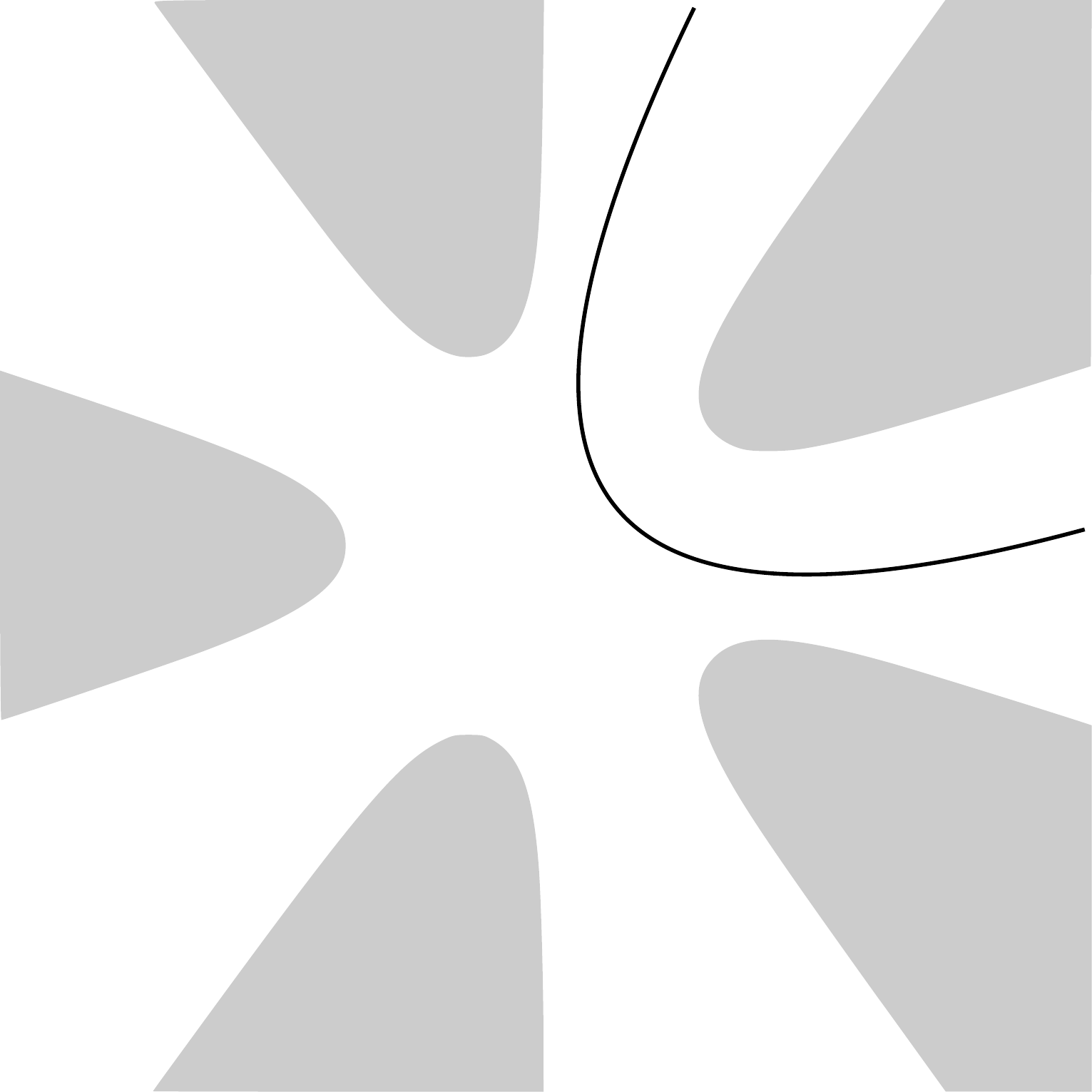}
   \put(80,35){$S_1$}
   \put(60,80){$S_2$}
   \put(5,70){$S_3$}
   \put(5,10){$S_4$}
   \put(55,5){$S_5$}
\end{overpic}
\caption*{$\mathcal P_0=\{ \{1,2\} \}$}
\end{subfigure}%
\begin{subfigure}{.3\textwidth}
   \centering
\includegraphics[scale=0.2]{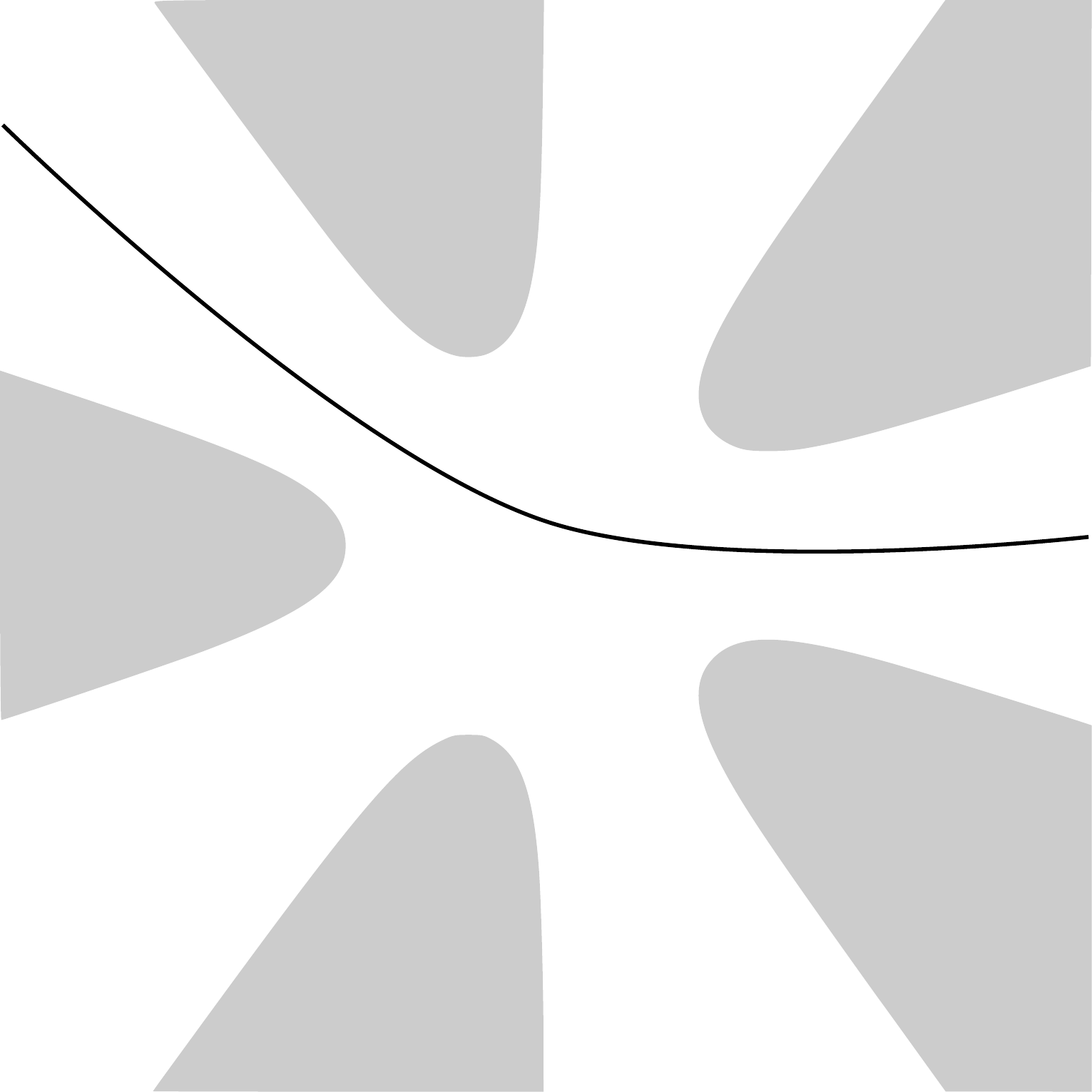}
\caption*{$\mathcal P_0=\{ \{1,3\}\}$}
\end{subfigure}%
\begin{subfigure}{.3\textwidth}
   \centering
\includegraphics[scale=0.2]{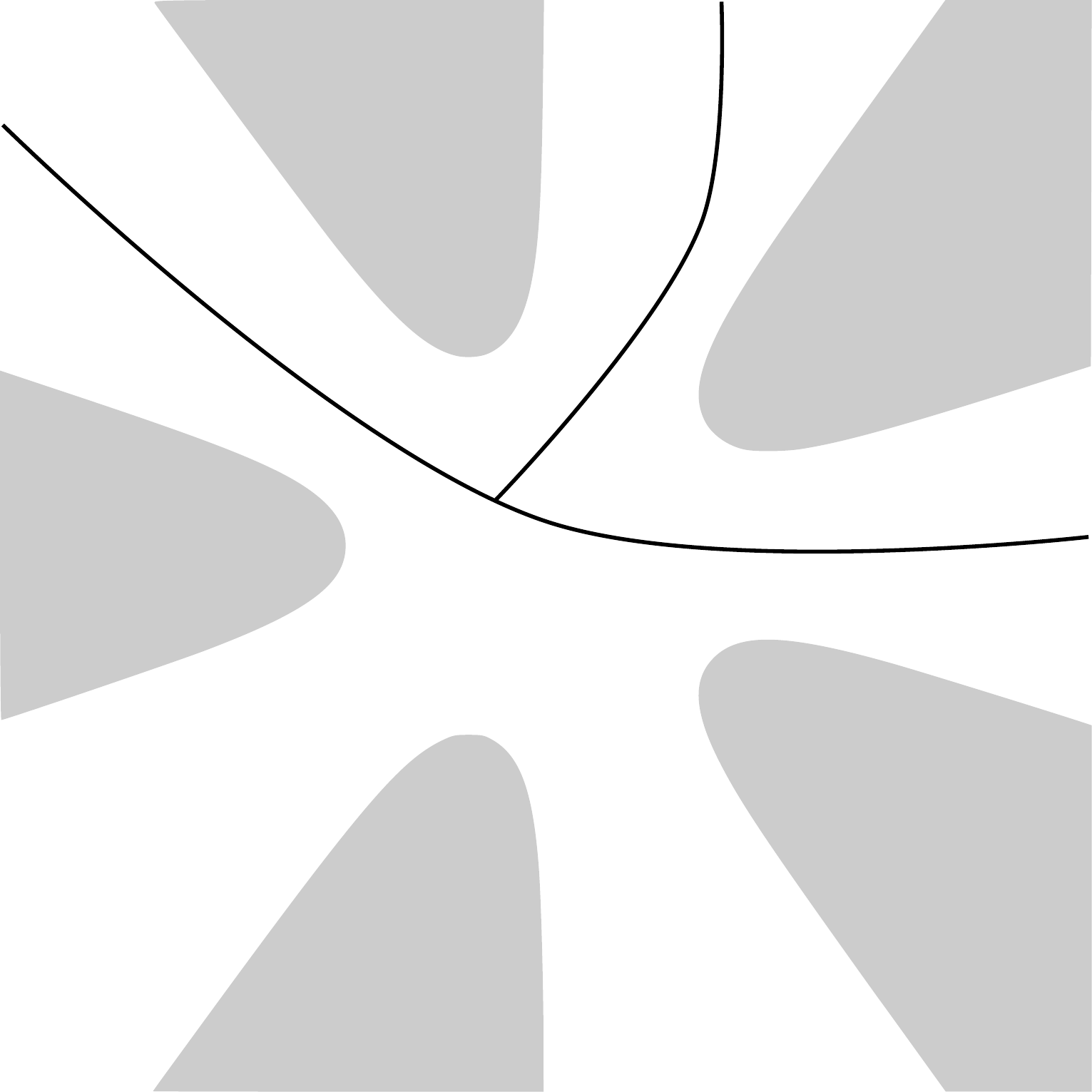}
\caption*{$\mathcal P_0=\{ \{1,2,3\}  \}$}
\end{subfigure}\\
\vspace{0.5cm}
  \begin{subfigure}{.3\textwidth}
   \centering
\includegraphics[scale=0.2]{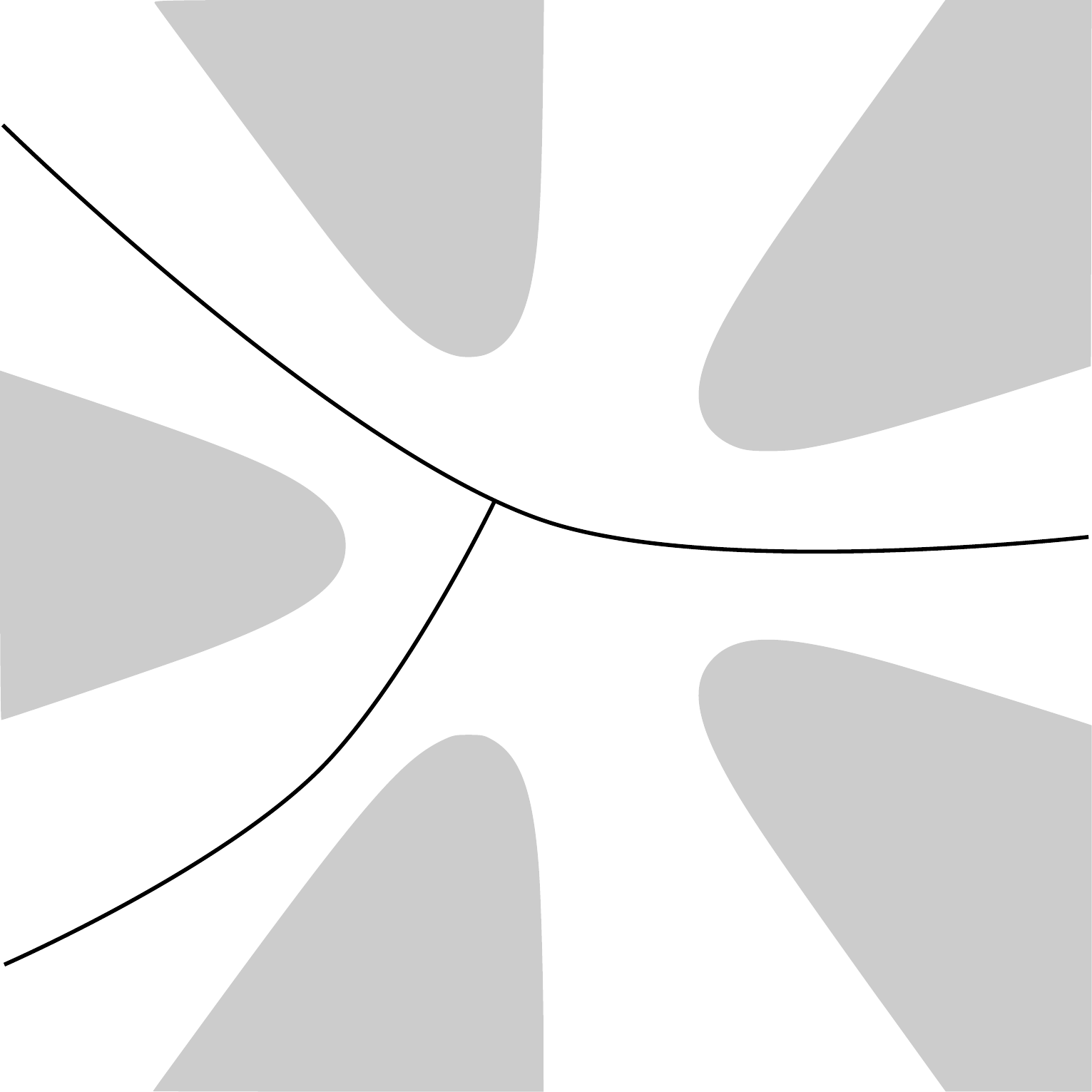}
\caption*{$\mathcal P_0=\{ \{1,3,4\}  \}$}
\end{subfigure}%
\begin{subfigure}{.3\textwidth}
   \centering
\includegraphics[scale=0.2]{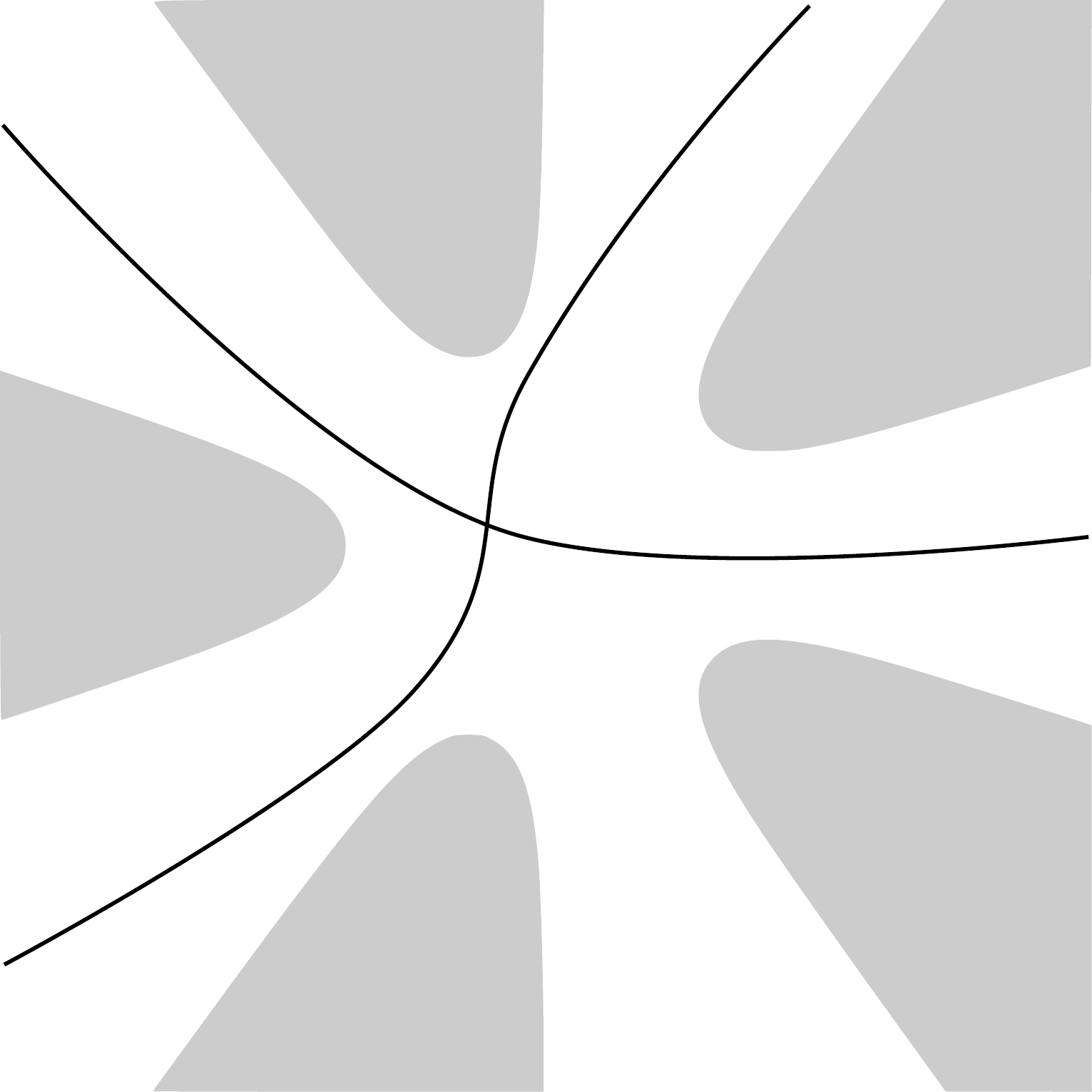}
\caption*{$\mathcal P_0=\{ \{1,2,3,4\}  \}$}
\end{subfigure}%
\begin{subfigure}{.3\textwidth}
   \centering
\includegraphics[scale=0.2]{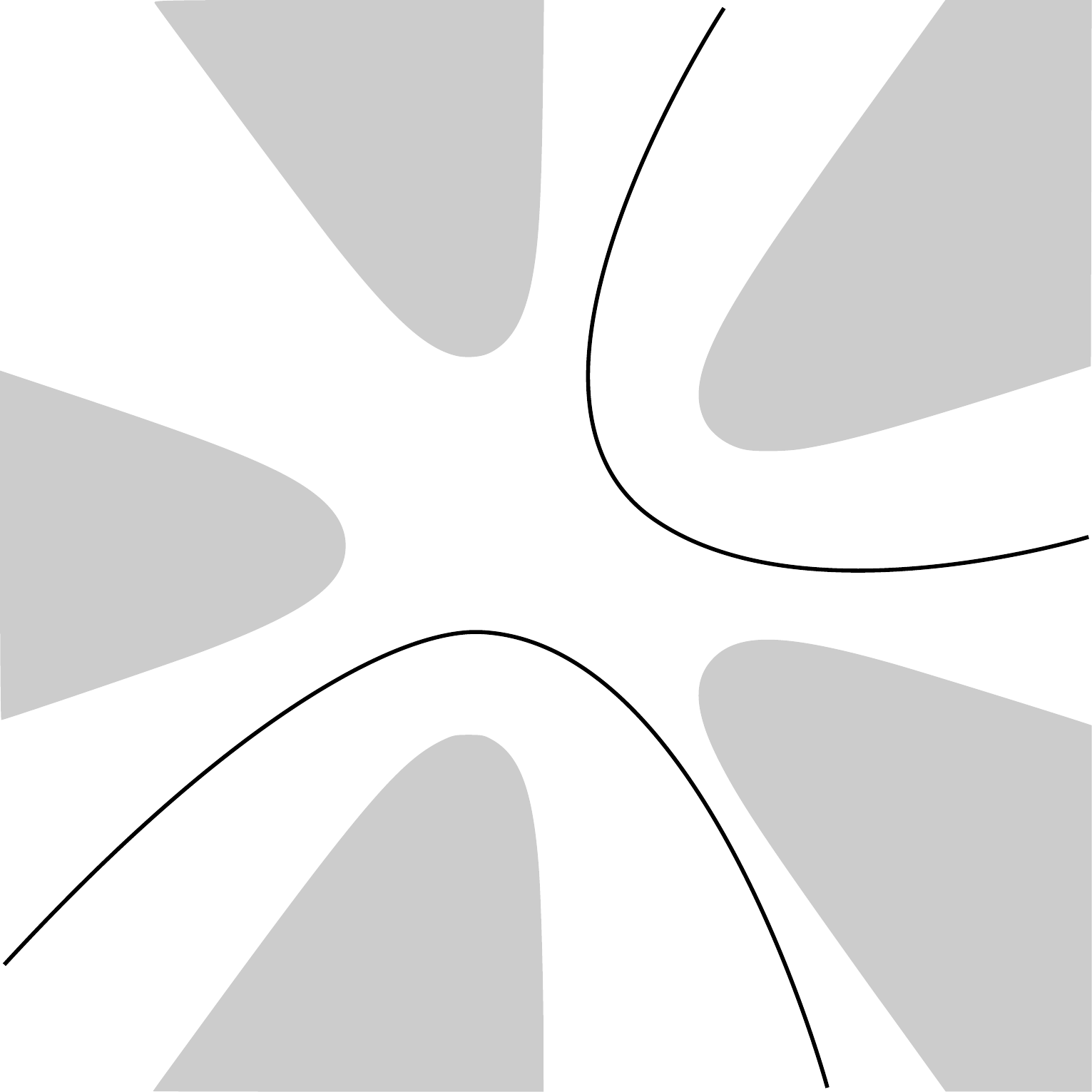}
\caption*{$\mathcal P_0=\{ \{1,2\},\{4,5\}  \}$}
\end{subfigure}\\%
\vspace{0.5cm}
  \begin{subfigure}{.3\textwidth}
   \centering
\includegraphics[scale=0.2]{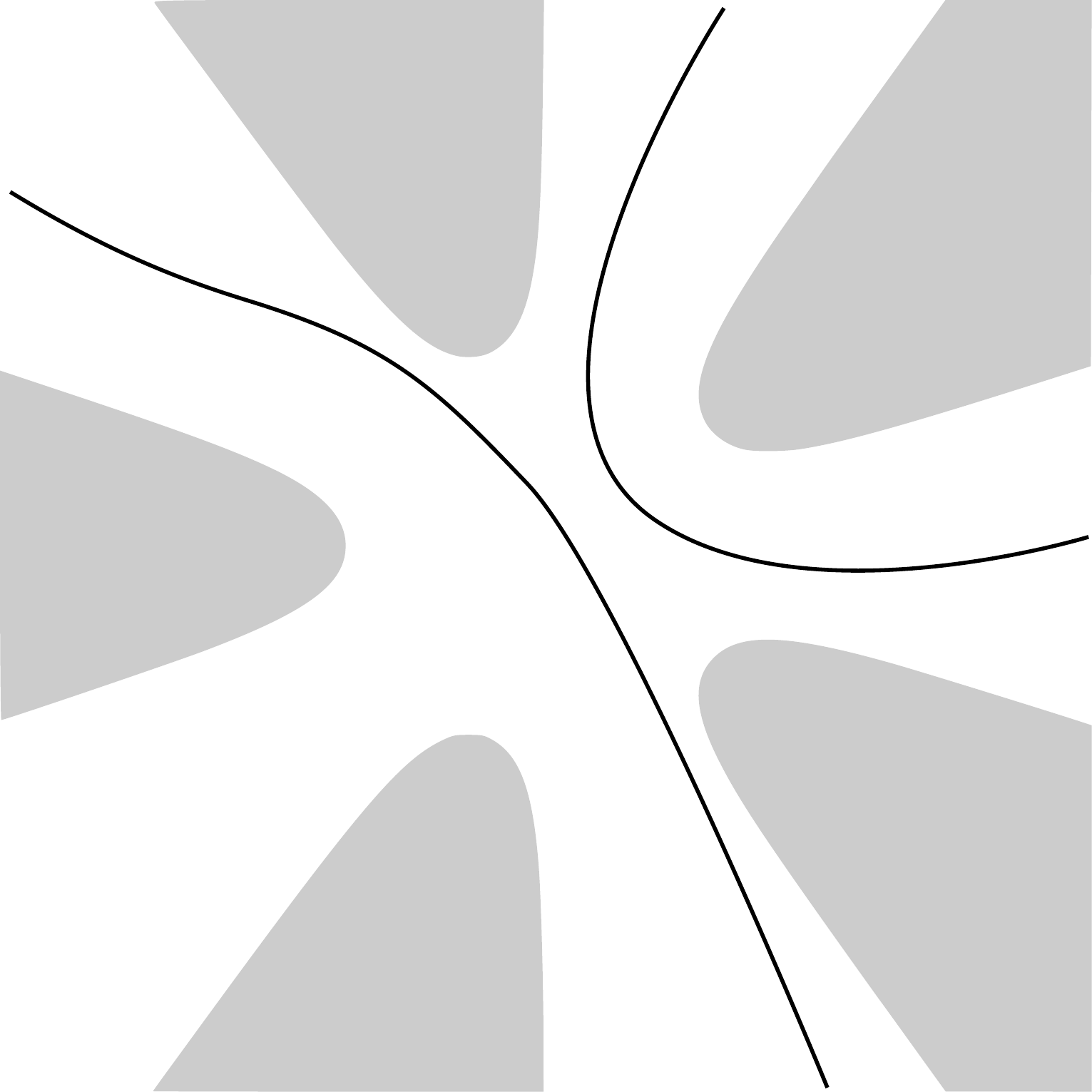}
\caption*{$\mathcal P_0=\{ \{1,2\},\{3,5\}  \}$}
\end{subfigure}%
\begin{subfigure}{.3\textwidth}
   \centering
\includegraphics[scale=0.2]{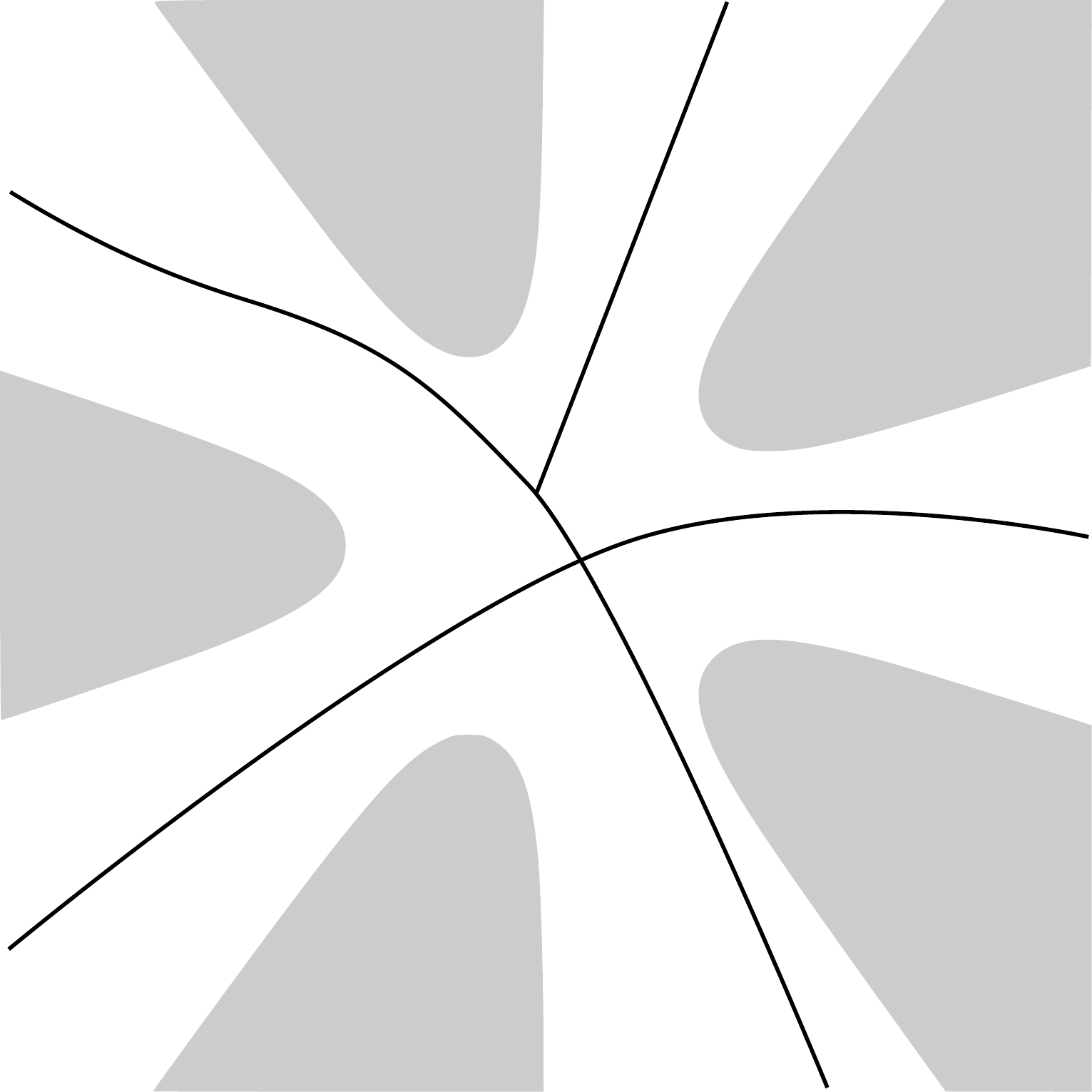}
\caption*{$\mathcal P_0=\{ \{1,2,3,4,5\}  \}$}
\end{subfigure}%
\caption{All possible choices of $\mathcal P$ and representative
contours in
$\mathcal T(\mathcal P)$ for a polynomial $V$ of degree $5$ and positive
leading coefficient. The sectors $S_1,\hdots,S_5$ are indicated on the
top left figure. In gray is represented the region where $\varphi$ is very negative.}
\label{table_configurations}
\end{figure}
%\end{center}

Figure \ref{table_configurations} shows all noncrossing partitions (up to rotation)
for a polynomial of degree $N =5$ and representative contours $\Gamma$ from each 
of the collections $\mathcal T(\mathcal P)$.

The restriction to noncrossing partitions in Definition \ref{definition_admissible_sets} 
is not essential. We could as well have defined the admissible contours for
arbitrary partitions, but then the class of admissible contours obtained for a given 
crossing partition would coincide with the class of admissible contours for some noncrossing
partition. 

\subsection{The main theorem}

For the polynomial $V$ and a fixed noncrossing partition $\mathcal P$ as above, 
the {\it max-min energy problem} for the pair $(V,\mathcal T(\mathcal P))$ asks for finding 
a contour $\Gamma_0 \in \mathcal T$ for which $I^\varphi(\cdot)$, see \eqref{definition_energy_functional},
attains its maximum on $\mathcal T$.
That is, 
\begin{align} I^{\varphi}(\Gamma_0) & = \max_{\Gamma \in \mathcal T} I^{\varphi}(\Gamma)
	= \max_{\Gamma \in \mathcal T} \min_{\mu \in \mathcal M_1^{\varphi}(\Gamma)} I^{\varphi}(\mu).
	\label{max-min-problem}
	\end{align}
	
In what follows we use
\begin{equation}
C^{\mu}(z)=\int\frac{d\mu(x)}{x-z}
\end{equation}
to denote the Cauchy Transform of a  measure $\mu$.
The main result of this paper is the following.

\begin{thm}\label{main_result}
Let $V$ be a polynomial of degree $N\geq 2$ and let $\mathcal P$ be a noncrossing partition with
class $\mathcal T(\mathcal P)$ of admissible contours as in Definition \ref{definition_admissible_sets}.
Then there exists a solution $\Gamma_0 \in \mathcal T(\mathcal P)$ to the max-min energy problem 
\eqref{max-min-problem} for $(V,\mathcal T(\mathcal P))$.

The contour $\Gamma_0$ has the S-property in the external field $\varphi=\re V$.

The equilibrium measure $\mu_0 := \mu^{\varphi,\Gamma_0}$ of $\Gamma_0$ in the external field $\varphi$ is 
supported on a finite union of 
analytic arcs that are critical trajectories of the quadratic differential $-R(z)dz^2$, where
\begin{equation} \label{polynomialR}
R(z)=\left( C^{\mu_0}(z) +\frac{V'(z)}{2}\right)^2,\quad z\in\C\setminus\supp\mu_0,
\end{equation}
is a polynomial of degree $2N-2$.

Furthermore $\mu_0$ is absolutely continuous with respect to arclength and
\begin{equation} \label{densitymu0}
d\mu_0(s)=\frac{1}{\pi i}R_+(s)^{1/2}ds.
\end{equation}
\end{thm}
See Section \ref{section_quadratic_differentials} for a brief discussion on quadratic differentials
and their critical trajectories.

In \eqref{densitymu0} we use a complex line element $ds$ on $\Gamma_0$, which induces an orientation
on $\Gamma_0$. Then  $R_+(s)^{1/2}$ denotes
the limiting value of 
\[ R(z)^{1/2} =  \int\frac{d\mu_0(x)}{x-z} +\frac{V'(z)}{2} \]
 as $z \to s \in \Gamma_0$ from the left-hand side. 

The contour $\Gamma_0$ that is the maximizer for the  max-min energy property in Theorem \ref{main_result} 
is not unique. We can modify the contour $\Gamma_0$ outside the support of $\mu_0$ slightly,
preserving the equilibrium measure in the external field. In this way we obtain a new contour
that also solves the max-min problem for $(V,\mathcal{T}(\mathcal{P}))$.

Although the contour is not unique, the equilibrium measure in the external field $\mu_0$ and its support
turn out to be uniquely determined by $V$ and $\mathcal P$. We do not know of a direct way to prove
this, but there is an indirect way using orthogonal polynomials. Assume for simplicity that
$\mathcal P_0 = \{ \{i,j\} \}$ for some $1 \leq i < j \leq N$. This means that any contour
$\Gamma\in \mathcal T(\mathcal P)$ connects two fixed distinguished sectors at infinity where $\re V(z)\to +\infty$. 

The (non hermitian) bilinear form defined on polynomials $p$ and $q$,
$$
\langle p,q \rangle=\langle p,q\rangle_{V,n} =\int_\Gamma p(z)q(z)e^{-nV(z)}dz
$$
is then well defined and independent of the choice of contours $\Gamma\in\mathcal T(\mathcal P)$. 
A polynomial $p_n$ of degree at most $n$ is {\it orthogonal} with respect to $\langle \cdot,\cdot\rangle$ if
$$
\langle p_n,q\rangle=0,
$$
for any polynomial $q$ of degree $\leq n-1$. Associated to the polynomial $p_n$ is its 
normalized zero counting measure $\chi_n$
$$
\chi_n=\frac{1}{\deg p_n}\sum_{p_n(z)=0}\delta_z.
$$ 
In studying $(\chi_n)$, Gonchar and Rakhmanov proved

\begin{thm}[{\cite[Theorem~3]{gonchar_rakhmanov_rato_rational_approximation}}] \label{gonchar_rakhmanov_theorem_zeros}
If there exists $\Gamma\in \mathcal T(\mathcal P)$ with the S-property in the external field $\varphi=\re V$
and $\C\backslash \supp\mu^{\varphi,\Gamma}$ is connected, then
$$
\chi_n\stackrel{*}{\to}\mu^{\varphi,\Gamma}.
$$
\end{thm}
In other words, if a contour $\Gamma\in \mathcal T(\mathcal P)$ has the S-property, then its equilibrium measure in the
external field is the weak limit of $(\chi_n)$. 

By Theorem \ref{main_result} any contour $\Gamma \in \mathcal T(\mathcal P)$ with 
the max-min property has the S-property. Also $\C\setminus\supp\mu^{\varphi,\Gamma}$ is connected, 
since otherwise we could remove self-intersecting subarcs of $\Gamma$ and increase its energy.
Thus, if $\Gamma_0$, $\Gamma_1$ are two solutions to the max-min problem $(V,\mathcal T(\mathcal P))$,
then by Theorem \ref{gonchar_rakhmanov_theorem_zeros} both equilibrium measures $\mu^{\varphi,\Gamma_0}$ 
and $\mu^{\varphi,\Gamma_1}$ are the weak limit of the sequence $(\chi_n)$. By uniqueness of limits, it follows that
$\mu^{\varphi,\Gamma_0}=\mu^{\varphi,\Gamma_1}$.

The above argument applies if $\mathcal P_0 = \{ \{i,j \}\}$ for some $1 \leq i < j \leq N$. 
For other choices of $\mathcal P_0$, we need to modify the bilinear form $\langle \cdot,\cdot\rangle$ to
get a right family of orthogonal polynomials $(p_n)$. For more details, we refer 
to \cite{bertola_boutroux, bertola_mo_spinor}.

\subsection{Quadratic differentials} \label{section_quadratic_differentials}

A polynomial $R$ of degree $2N-2$, $N\geq 2$, defines on the Riemann sphere $\overline\C$ the 
{\it quadratic differential} $-R(z)dz^2$. A basic reference for quadratic
differentials is the book by Strebel \cite{strebel_book}.

The {\it critical points} of $-R(z)dz^2$ are the zeros of $R$, whose orders as critical points are the same as 
their respective orders as zeros of $R$, and the point $z=\infty$, which is
a pole of order $2N+2$. The other points in $\overline\C$ are regular points of $-R(z)dz^2$.

An arc $\gamma$ is called a {\it horizontal arc} of $-R(z)dz^2$ if, along $\gamma$
$$
-R(z)dz^2>0.
$$
At a regular point, this condition can be locally stated as
$$
\re \int^z \sqrt{R(s)} \, ds \equiv \mathrm{const},\quad z\in \gamma.
$$
A maximal horizontal arc is called a {\it horizontal trajectory}, or shortly {\it trajectory}, of $-R(z)dz^2$, and a trajectory that ends at a zero of $R$ is called a {\it
critical trajectory}. 

Also of interest are the {\it vertical trajectories}, which are defined instead by the requirement
$$
	R(z)dz^2 >0,
$$
which locally at a regular point is equivalent to
$$
\im \int^z \sqrt{R(s)} \, ds \equiv \mathrm{const}.
$$

The local structure of trajectories is well understood. Through any regular point $z_0$ passes 
exactly one horizontal and one vertical trajectory, which are analytic arcs that intersect at $z_0$ only,
where they are orthogonal.

From a zero of $-R(z)dz^2$ of order $p$ emanate exactly $p+2$ trajectories, at equal
consecutive angles $\frac{2\pi}{p+2}$.

At the pole $z=\infty$ of order $2N+2$, the local behavior of trajectories is as 
follows \cite[Theorem~7.4]{strebel_book}. There are $2N$ directions at $\infty$ forming equal
angles, such that any trajectory ending at $\infty$ does so in one of these directions. 
There is a neighborhood $U$ of $\infty$ such that any trajectory entering $U$ ends at
$\infty$ in at least one direction, and trajectories entirely contained in $U$ end at $\infty$ 
in two consecutive directions, see Figure \ref{local_trajectories}, which shows the trajectories
at infinity for a pole of order $10$ after inversion $z \mapsto \frac{1}{z}$.

% \begin{center}
\begin{figure}[t]
\begin{overpic}[scale=.4]{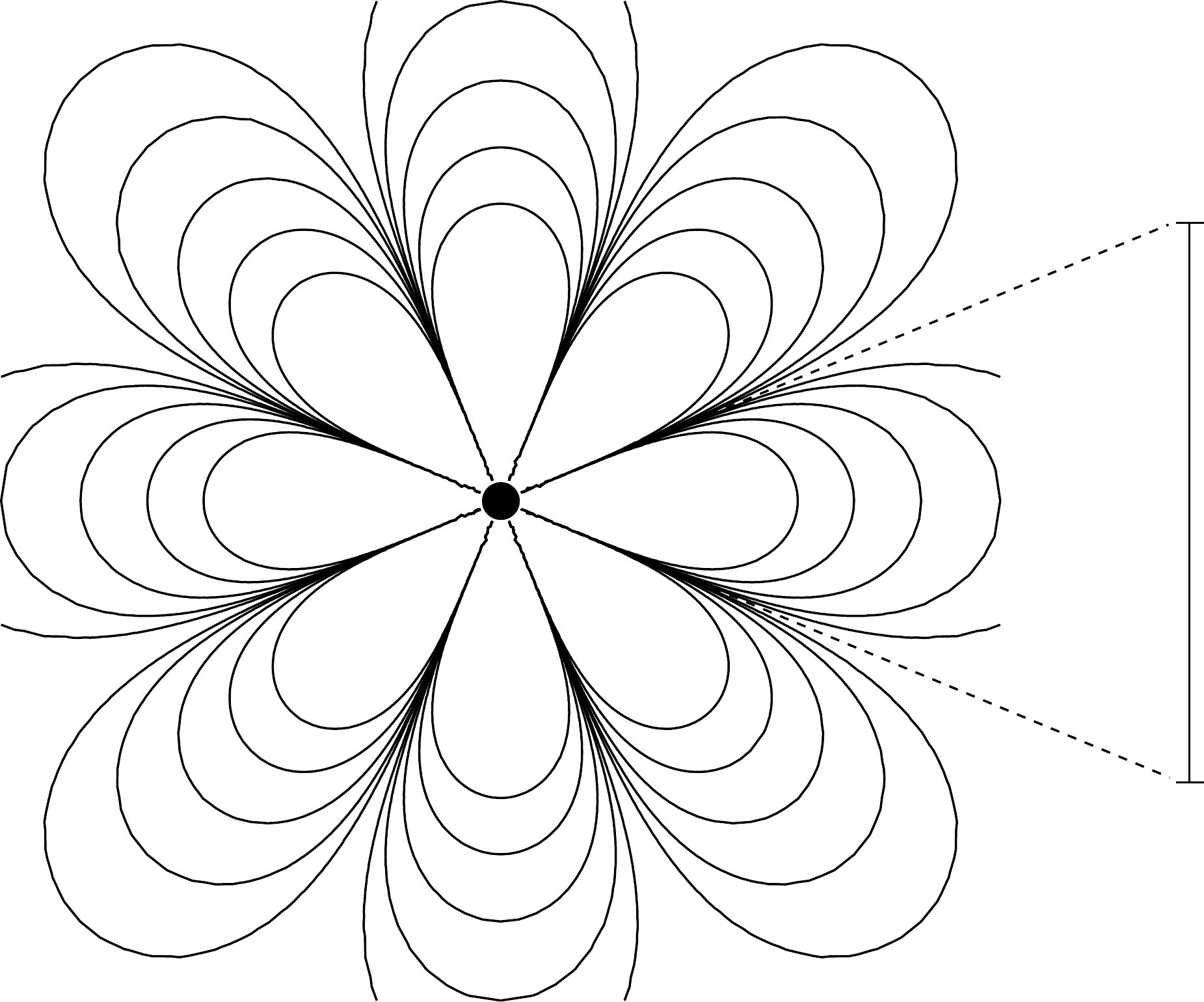}
  \put(220,80){$\theta=\frac{\pi}{4}$}
\end{overpic}
\caption{Trajectories of a quadratic differential near a pole of order 10. The trajectories
end at the pole in eight possible directions.}\label{local_trajectories}
\end{figure}
% \end{center}

\subsection{Rakhmanov's results and overview of the rest of the paper}

As mentioned in the Introduction, our work is very much influenced by the recent work of 
Rakhmanov \cite{rakhmanov_orthogonal_s_curves}, which we now explain.

The first result obtained by Rakhmanov in \cite{rakhmanov_orthogonal_s_curves}, which
is relevant to us is the following.
\begin{thm}[{\cite[Theorem~9.2]{rakhmanov_orthogonal_s_curves}}] \label{rakhmanov_theorem_1}
 Suppose $F$ is closed in $\overline \C$. If
 $$
 -\infty<I^\varphi(F)<+\infty,
 $$
 then the equilibrium measure $\mu^{\varphi,F}$ in external field $\varphi$ exists and is unique.
\end{thm}

In contrast with previous results, Theorem \ref{rakhmanov_theorem_1} does not contain
any assumption on the growth of $\varphi$ at $\infty$. Actually the theorem in \cite{rakhmanov_orthogonal_s_curves}
is more general as it also deals with situations where the external field is continuous 
except at a finite number of points, and no growth restriction near these singularities.

Rakhmanov considers the energy functional
$$
F\mapsto I^\varphi(F)
$$
defined on a metric space of compacts, which we explain next.

The Riemann sphere $\overline \C$ is mapped to the sphere $\mathcal S\subset \R^3$ centered at $(0,0,1/2)$ 
with radius $\frac{1}{2}$,
$$
\mathcal S=\left\{ (x_1,x_2,x_3)\in\mathcal \R^3 \mid x_1^2+x_2^2+(x_3-\tfrac{1}{2})^2=\tfrac{1}{4} \right\},
$$
via the inverse stereographic map $L:\overline \C\to \mathcal S$ given by $L(\infty)=(0,0,1)$ and
$$
L(z)=\left(\frac{\re z}{1+|z|^2},\frac{\im z}{1+|z|^2},\frac{|z|^2}{1+|z|^2}\right), \qquad z\in \C.
$$

The {\it hyperbolic metric} on $\overline \C$ is then given by
$$
\distp(z,w)=\|L(z)-L(w)\|,
$$
where the norm $\|\cdot\|$ on the right-hand side above is the Euclidean norm in $\mathcal \R^3$. 
The following relation holds
$$
\distp(z,w)=\frac{|z-w|}{\sqrt{1+|z|^2}\sqrt{1+|w|^2}}, \quad z,w\in\C.
$$

This metric induces a metric on compact subsets of $\overline \C$. If $K_1,K_2\subset \overline \C$ 
are compacts, the {\it Hausdorff metric} between them is given by the formula
\begin{equation} \label{Hausdorffmetric}
\disth(K_1,K_2)=\max\left\{ \sup_{x\in K_1}\dist(x,K_2), \sup_{y\in K_2}\dist(y,K_1)  \right\},
\end{equation}
where the distance between points and sets above is the usual one induced by $\distp$,
$$
\dist(x,K)=\inf_{y\in K}\distp(x,y).
$$
Alternatively, denoting
$$
(K)_\delta=\{x\in \overline{\C} \mid \ \dist(x,K)<\delta\},
$$
it can be equivalently defined by
$$
\disth(K_1,K_2)=\inf \left\{ \delta>0 \mid K_1\subset (K_2)_\delta, \ K_2\subset (K_1)_\delta
\right\}.
$$

The Hausdorff metric is thus defined on closed sets of $\overline\C$. Naturally it induces a 
metric on closed sets of $\C$ by adding the point at infinity to unbounded closed sets. 
In what follows, when we talk about the Hausdorff metric on closed 
subsets of $\C$, we mean the metric obtained with this natural identification to closed
subsets of $\overline\C$.

We equip $\mathcal T$ with the Hausdorff metric \eqref{Hausdorffmetric}
induced by the hyperbolic distance. So whenever we refer to a topological notion 
on $\mathcal T$ such as convergence or continuity, we will always mean with respect
to the Hausdorff metric just introduced.

Suppose $\mathcal F$ is a class of closed subsets of $\C$ for which there exists 
an absolute constant $c$ such that every compact set in $\mathcal F$ has at most $c$ connected
components. Equip $\mathcal F$ with the Hausdorff metric just introduced. 
For such classes $\mathcal F$, Rakhmanov proves
the following.

\begin{thm}[{\cite[Theorem~3.2]{rakhmanov_orthogonal_s_curves}}]\label{theorem_upper_semicontinuity}
The energy functional $I^\varphi:\mathcal F\to [-\infty,+\infty]$ is upper semicontinuous.
\end{thm}

If the class $\mathcal F$ is closed, then it follows from \eqref{theorem_upper_semicontinuity}
there exists a set $F_0\in \mathcal F$ maximizing $I^\varphi$. 
However, it is easy to see that the class of admissible contours $\mathcal T = \mathcal T(\mathcal P)$ 
introduced in Definition \ref{definition_admissible_sets} is not closed in the Hausdorff metric. 
An easy way to see this is through space-filling curves. One can approximate sets of large
planar measure with smooth contours, and the limiting set is not a smooth contour anymore.

A natural idea would then be to work with the closure $\overline{\mathcal T}$ of the class $\mathcal T$
in the Hausdorff metric. There is however an issue with that, since sets in 
$\overline{\mathcal T}$ may stretch out to infinity in unlikely directions, and we might then recover the S-curve for
the admissible class $\mathcal T(\mathcal P')$ corresponding to another choice of partition $\mathcal P'$.

Our technique to overcome this last issue is to introduce a suitable subclass $\mathcal T_M \subset \mathcal T$, 
see \eqref{definition_subclass_cut_off}, consisting of contours that do not enter a ``forbidden region'' $\Delta_M$.
Then we take the closure
$\mathcal F_M=\overline{\mathcal T}_M$, see also \eqref{definition_FM}. 
For the closed class $\mathcal F_M$, Theorem \ref{theorem_upper_semicontinuity} assures us that there
exists a set $F_0$ maximizing the energy functional $I^\varphi(\cdot)$ on $\mathcal F_M$. 
Moreover, with $M$ chosen large enough,  we can guarantee that the set $F_0$ does not come to the boundary of
the forbidden region $\Delta_M$, which implies that small deformations of $F_0$ in any direction also belong to $\mathcal F_M$.

Any $C^2_c$ function $h:\C \to \C$ generates a one-parameter
deformation $F_0^t$ of this maximizer set by
\begin{equation}
F_0^t=\{z+th(z) \mid z\in F \}, \qquad t\in\R.
\end{equation}
Now, since $F_0^t$ also belongs to $\mathcal F$ for $t$ sufficiently close to $0$, 
we have that (the limit actually exists)
$$
\lim_{t\to 0}\frac{I^\varphi(F_0^t)-I^\varphi(F_0)}{t} = 0.
$$
because $F_0$ is a maximizer set. 
Rakhmanov \cite{rakhmanov_orthogonal_s_curves} then proves that this limit implies 
that the equilibrium measure $\mu^{\varphi, F_0}$ in the external field $\varphi$ is critical, in the sense
introduced by Mart\'inez-Finkelshtein and Rakhmanov \cite{martinez_rakhmanov}, see also 
Section \ref{critical_measures} below. 

As it was already observed by Mart\'inez-Finkelshtein and Rakhmanov \cite{martinez_rakhmanov}, 
the Cauchy transform of a critical measure is an algebraic function. In the present
context, this means that the Cauchy transform $C^{\mu_0}$ of the equilibrium measure in external field
satisfies an algebraic equation of the form 
$$
\left(C^{\mu_0}(z)+\frac{V'(z)}{2}\right)^2=R(z), \qquad \text{for a.e. } z \in \mathbb C,
$$
where $R$ is a polynomial of degree $2n-2$, see Proposition \ref{proposition_algebraic_equation_critical_measure} below. 
The main technical difference between \cite{martinez_rakhmanov} and the present work is that in \cite{martinez_rakhmanov} 
this is obtained under the assumption that the critical measure  is compactly
supported, while here we obtain the algebraic equation first, and as one of its consequences we find 
that $\mu_0$ must be compactly supported. As in \cite{martinez_rakhmanov}, the
algebraic equation also implies that the measure $\mu_0$ is supported on a finite union of analytic arcs and 
its density can be recovered from the polynomial $R$, see Propositions
\ref{proposition_density_critical_measure}, \ref{proposition_support_critical_measure}. 
Also, the S-property ultimately follows from the algebraic
equation, see Corollary \ref{corollary_s_property}.

As a final step, once some properties of the equilibrium measure in external field $\mu_0$  of the maximizer 
set $F_0$ are known,  we drop the parts of $F_0$ that do not belong to $\supp\mu_0$, and
we give in Section \ref{proof_main_result} a construction of a curve $\Gamma_0\in \mathcal T(\mathcal P)$ 
whose equilibrium measure in the external field $\varphi$ 
is the desired measure $\mu_0$.

The rest of the paper is organized as follows. In Section \ref{critical_measures} we discuss critical measures. 
The results therein presented are similar to the ones  obtained for the first time by
Mart\'inez-Finkelshtein and Rakhmanov \cite{martinez_rakhmanov}. However, as mentioned before and further explained later, 
the setup in \cite{martinez_rakhmanov} is not the same as ours, so we preferred
to present detailed proofs in Section \ref{critical_measures}. 
In Section \ref{critical_measures_critical_sets} we explain the link between critical measures and
critical sets, following the lines of \cite{rakhmanov_orthogonal_s_curves}. 
Section \ref{proof_main_theorem} is devoted to the proof of Theorem~\ref{main_result}.

%%%%%%%%%%%%%%%%%%%%%%%%%%%%%%%%%%%%%%%%%%%%%%%%%%%%%%%%%%%%%%%%%%%%%%
\section{Critical Measures}\label{critical_measures}
%%%%%%%%%%%%%%%%%%%%%%%%%%%%%%%%%%%%%%%%%%%%%%%%%%%%%%%%%%%%%%%%%%%%%%

As before, $V$ is a polynomial of degree $N \geq 2$ and the external field 
$\varphi$ is given by
\begin{equation}\label{external_field_real_part_polynomial}
	\varphi(z)=\re V(z),\qquad z\in\C.
\end{equation}
With proper modifications, all the results and proofs in this section are also 
valid for external fields given by the real part of an analytic function, possibly
multivalued, with rational (and so single-valued) derivative, but for our purposes 
it is enough to consider the polynomial case. 

The results we are going to show here were proven before by 
Mart\'inez-Finkelshtein and Rakhmanov \cite{martinez_rakhmanov} 
under the assumption that the measures involved are a priori compactly
supported and the external field has a rational derivative with simple poles. The techniques 
we are using are also similar, but sometimes simplified and adapted to our
needs. The main difference here is that we do not impose a priori that a critical measure must 
have compact support. Ultimately, we do find that critical measures must
be compactly supported (for external fields as in \eqref{external_field_real_part_polynomial}), 
and then the results here are the same as the ones obtained in \cite{martinez_rakhmanov}.

\subsection{Derivative of the energy functional}

Consider the set of test functions $C^2_c$, consisting of $C^2$ complex-valued functions in $\C$ with compact support.
For a function $h\in C_c^2$ and $t\in \R$ denote
\begin{equation}\label{h_variation}
h_t(z)=z+th(z),\quad z\in \C.
\end{equation}
Any function $h \in C^2_c$ generates local variations of sets $F\mapsto F^t=h_t(F)$ and also 
of measures $\mu\mapsto \mu^t$, where $\mu^t$ is the pushforward measure induced by $h_t$. If
$\mu$ has finite weighted logarithmic energy then the same is true for $\mu^t$.

For a measure $\mu\in M_1^\varphi(\C)$ and $h\in C_c^2$, we define the {\it derivative} of the weighted energy functional at $\mu$ in the
direction of $h$ by
\begin{equation} \label{equation_derivative_energy_measure}
D_hI^\varphi(\mu)=\lim_{t\to 0} \frac{I^\varphi(\mu^t) - I^\varphi(\mu)}{t},
\end{equation}
whenever the limit exists.
A first result is that $D_hI^\varphi(\mu)$ indeed always exists.

\begin{prop}\label{proposition_formula_derivative} 
Given $\mu \in \mathcal{M}_1^\varphi(\C )$ and $h\in C_c^2$, the derivative $D_hI^\varphi(\mu)$ exists and is given by
\begin{equation}\label{explicit_formula_derivative}
 D_hI^\varphi(\mu)=-\re\left( \iint \frac{h(x)-h(y)}{x-y} \, d\mu(x)d\mu(y) - \int V'(x)h(x) \, d\mu(x) \right).
\end{equation}
\end{prop}

\begin{proof}
The proof here is the same as in \cite[Proof of Lemma 3.7]{martinez_rakhmanov}. By the Change of Variables Formula
$$
I^\varphi(\mu^t)=\iint \log \frac{1}{\left|x-y +t(h(x)-h(y))\right|} \, d\mu(x)d\mu(y)+\int \varphi(x+th(x)) \, d\mu(x),
$$
which implies
\begin{align}\nonumber
 I^\varphi(\mu^t)-I^\varphi(\mu) & =  -\iint \log\left| 1+t\frac{h(x)-h(y)}{x-y} \right| \, d\mu(x)d\mu(y) \\
		     & \quad    +\int \left(\varphi(x+th(x))-\varphi(x)\right) \, d\mu(x)  \label{equation_for_derivative}.
\end{align}

Since $h\in C_c^2$, we have for $t \to 0$,
\begin{align}\nonumber
\log\left| 1+t\frac{h(x)-h(y)}{x-y} \right| & = \re \left[\log\left( 1+t\frac{h(x)-h(y)}{x-y} \right)\right] \\
					    & =  t \re \left(\frac{h(x)-h(y)}{x-y}\right)+\Boh(t^2), \label{equacao_auxiliar_9}
\end{align}
with the implicit constant in $\Boh(t^2)$ uniform in $x,y \in \C$, and analogously
$$
V(x+th(x))-V(x)=tV'(x)h(x)+\Boh(t^2),
$$
where again the $\Boh$ term is uniform in $x \in \C$. This implies
$$
\varphi(x+th(x))-\varphi(x)=t\re\left(h(x)V'(x)  \right)+\Boh(t^2) \quad \text{ as } t \to 0,
$$
and the result follows by plugging this last equation and \eqref{equacao_auxiliar_9} 
into \eqref{equation_for_derivative}.
\end{proof}

\subsection{Critical measures}

\begin{definition}\label{definition_critical_measures}
We say a measure $\mu \in \mathcal M_1^\varphi(\C)$ is $\varphi${\it -critical} if
$$
D_hI^\varphi(\mu)=0
$$
for every $h\in C^2_c$.
\end{definition}
The definition of $\varphi$-critical measures introduced in \cite{martinez_rakhmanov} allows 
the external field to have singularities. In that situation, the
test functions $h$ appearing in Definition~\ref{definition_critical_measures} should also vanish 
on a set $\mathcal A$ of zero capacity containing the singular points
of $\varphi$, leading to the notion of $(\varphi,\mathcal A)$-critical measures. In \cite{martinez_rakhmanov} 
it is also imposed that critical measures must have compact support.
In the case considered here the only singularity of the external field is at $z=\infty$, and for 
convenience we are initially considering  functions that are not just vanishing at infinity but
also in a neighborhood of it, that is, compactly supported functions.

Considering $h$ and $ih$ separately, we easily obtain the following from Proposition~\ref{proposition_formula_derivative}.
\begin{cor}\label{corollary_integral_equation_critical_measure}
 A measure $\mu\in \mathcal M_1^\varphi(\C)$ is $\varphi$-critical if, and only if,
\begin{equation}\label{integral_equation_critical_measure}
\iint \frac{h(x)-h(y)}{x-y}d\mu(x)d\mu(y)=\int V'(x)h(x)d\mu(x)
\end{equation}
for every $h\in C_c^2$.
\end{cor}
Equation \eqref{integral_equation_critical_measure} appears frequently in the context 
of random matrices, although usually in different forms and names as for
example {\it Loop equations}, {\it Schwinger-Dyson equations} or {\it Ward identities}, see 
\cite{ameur_hedenmalm_makarov, duits_painleve_kernels, mclaughlin_ercolani_loop_equations}.

We need to extend \eqref{integral_equation_critical_measure} to larger classes of test
functions.

\begin{lem}\label{lemma_test_function_at_infinity}
Let $\mu$ be a $\varphi$-critical measure. Then equation 
\eqref{integral_equation_critical_measure} remains valid for $h\in C^2$
satisfying
 \begin{equation}\label{assumption_behavior_at_infinity_test_function}
 h(z)=\Boh\left(\frac{1}{z^{N-1}}\right),\quad z\to\infty.
 \end{equation}
\end{lem}
\begin{proof}
 Introduce the function $\Theta_n:[0,+\infty)\to \R$ by
$$
\Theta_n(t)= \begin{cases} 
 1, & t\leq n, \\
 \frac{1}{16n^5}(3n-t)^3(2n^2-3nt+3t^2) & n<t<3n, \\
 0, & t\geq 3n.
\end{cases}
$$
Then $\Theta_n \in C^2_c$, $\Theta_n(t)\stackrel{n\to\infty}{\to} 1$ and
\begin{equation}\label{equacao_auxiliar_26}
	0\leq \Theta_n(t)\leq 1,\qquad  |\Theta'_n(t)|< \frac{1}{n}, \quad t\geq 0.
\end{equation}

For $h\in C^2$ satisfying \eqref{assumption_behavior_at_infinity_test_function}, define
$$
h_n(x)=\Theta_n(|x|)h(x),\quad x\in \C.
$$
Then $h_n$ belongs to $C^2_c$, so that equation \eqref{integral_equation_critical_measure} is valid for it.
From assumption \eqref{assumption_behavior_at_infinity_test_function} and the fact that $\deg V = N$, we see
$$
V'(x)h(x)=\Boh(1), \quad x\to\infty,
$$ 
 and from the definition of $h_n$ and the Dominated Convergence Theorem, it is easily seen
\begin{equation}\label{equacao_auxiliar_25}
\int V'(x)h_n(x)d\mu(x)\stackrel{n\to\infty}{\to} \int V'(x)h(x)d\mu(x).
\end{equation}

The measure $\mu$ has finite logarithmic energy, so it has no mass points. In particular, the diagonal 
$\{(x,x) \mid x\in\C\}$ has zero $\mu\times\mu$
measure, implying
$$
\frac{h_n(x)-h_n(y)}{x-y}\stackrel{n\to\infty}{\to} \frac{h(x)-h(y)}{x-y},\quad \mu\times\mu \text{-a.e.}
$$
Using \eqref{equacao_auxiliar_26}, we get
\begin{align*}
\left| \frac{h_n(x)-h_n(y)}{x-y} \right| & =  
	\left| \frac{h(x)-h(y)}{x-y}\Theta_n(|x|)+\frac{\Theta_n(|x|)-\Theta_n(|y|)}{x-y}h(y)\right| \\
					  & \leq \left|\frac{h(x)-h(y)}{x-y}\right|+|h(y)|.
\end{align*}

From \eqref{assumption_behavior_at_infinity_test_function} we find that both terms on the right-hand side 
are bounded, and therefore $\mu\times\mu$-integrable. The Dominated
Convergence Theorem can be applied, yielding
$$
\iint \frac{h_n(x)-h_n(y)}{x-y} d\mu(x)d\mu(y)\stackrel{n\to\infty}{\to} \iint \frac{h(x)-h(y)}{x-y} d\mu(x)d\mu(y),
$$
and the equation \eqref{integral_equation_critical_measure} for $h$ follows from this last limit and
\eqref{equacao_auxiliar_25}, keeping in mind that \eqref{integral_equation_critical_measure} is valid for each $h_n$.
\end{proof}

\begin{lem}\label{lemma_test_function_singularity}
Let $\mu$ be a $\varphi$-critical measure. If $z \in \C$ is such that
 \begin{equation}\label{equacao_auxiliar_27}
 \int\frac{d\mu(x)}{|x-z|}<+\infty,
 \end{equation}
then \eqref{integral_equation_critical_measure} is valid for every function of the form
 $$
 h(x)=\frac{g(x)}{x-z}, \qquad g \in C_c^2.
 $$
\end{lem}
\begin{proof}
The same idea used to prove Lemma \ref{lemma_test_function_at_infinity} can also be used here, 
namely approximating $h$ by a sequence of functions for which
\eqref{integral_equation_critical_measure} is valid. But now the approximating sequence needs to be modified. 
By translating $x\mapsto x-z$ we can assume $z=0$, and then we rewrite
$$
h(x)=\frac{g(x)\overline x}{|x|^2}.
$$

Define
$$
h_n(x)=\frac{g(x)\overline x}{|x|^2+\epsilon_n^2},
$$
where $(\epsilon_n)$ is a sequence of positive numbers converging to zero. 
Clearly $h_n\in C^2_c$, and so equation \eqref{integral_equation_critical_measure} is valid for it.

Note that \eqref{equacao_auxiliar_27} assures us that $h$ is $\mu$-integrable. Proceeding then similarly as in the proof of Lemma \ref{lemma_test_function_at_infinity},
\begin{equation}\label{equacao_auxiliar_4}
\int V'(x)h_n(x)d\mu(x)\stackrel{n\to\infty}{\to}\int V'(x)h(x)d\mu(x), 
\end{equation}
and also
\begin{equation}\label{equacao_auxiliar_16}
\frac{h_n(x)-h_n(y)}{x-y}\stackrel{n\to\infty}{\to} \frac{h(x)-h(y)}{x-y},\quad \mu\times\mu \text{-a.e.}
\end{equation}
Moreover,
\begin{align}\nonumber
 \left| h_n(x)-h_n(y)  \right| & \leq 
		  \left| \frac{ 
				|y|^2\overline x g(x)-|x|^2\overline y g(y)  
				  }{ (|x|^2+\epsilon_n^2)(|y|^2+\epsilon_n^2)   } \right| 
				    + \frac{\epsilon_n^2 \left| \overline x g(x) - \overline y 
					      g(y) \right|}{(|x|^2+\epsilon_n^2)(|y|^2+\epsilon_n^2)} \\ \nonumber
			       & \leq 
		\left| \frac{ 
				|y|^2\overline x g(x)-|x|^2\overline y g(y)  
				  }{ |x|^2 |y|^2   } \right|\\ \nonumber
				 & \quad + \frac{ \left|\overline x g(x) - \overline y g(y) \right| }{|x||y|}
				    \frac{ 1 }{ \left(\frac{|x|}{\epsilon_n}+\frac{\epsilon_n}{|x|}\right) 
						\left(\frac{|y|}{\epsilon_n}+\frac{\epsilon_n}{|y|}\right)} \\
			      & \leq  \left| \frac{y g(x)-x g(y)}{xy}\right|+\left| \frac{\overline x g(x)-\overline y g(y)}{xy} \right|, 
\label{inequality_quotient_h}
\end{align}
where in the last inequality we used $t+1/t> 1$ for $t\geq 0$.

Now, using the trivial decompositions
\begin{align*}
 \frac{yg(x)-xg(y)}{xy(x-y)} & = \frac{1}{y}\frac{g(x)-g(y)}{x-y}-\frac{g(x)}{xy}, \\
 \frac{\overline x g(x)-\overline y g(y)}{\overline x\overline y(\overline x-\overline y)} & = \frac{1}{\overline x}\frac{g(x)-g(y)}{\overline
x-\overline y}+\frac{g(x)}{\overline x\overline y}
\end{align*}
 in \eqref{inequality_quotient_h}, we get
$$
\left| \frac{h_n(x)-h_n(y)}{x-y} \right| \leq \left(\frac{1}{|x|}+\frac{1}{|y|}\right)\left| \frac{g(x)-g(y)}{x-y} \right| + \frac{2}{|y|}\left|
\frac{g(x)}{x} \right|
$$

Since $g\in C^2_c$, the quotient
$$
\frac{g(x)-g(y)}{x-y},
$$
is bounded, as well as $g$, say both by $M$. This last inequality then implies
$$
\left| \frac{h_n(x)-h_n(y)}{x-y} \right| \leq 	\frac{M}{|x|}+\frac{M}{|y|}+\frac{2M}{|x||y|}.
$$
From assumption \eqref{equacao_auxiliar_27} the right-hand side above is $\mu\times\mu$-integrable. 
From the Dominated Convergence Theorem and \eqref{equacao_auxiliar_16} we conclude
$$
\iint \frac{h_n(x)-h_n(y)}{x-y} d\mu(x)d\mu(y)\stackrel{n\to\infty}{\to} \iint \frac{h(x)-h(y)}{x-y} d\mu(x)d\mu(y).
$$

Since \eqref{integral_equation_critical_measure} is valid for the $h_n$’s, this last limit together 
with \eqref{equacao_auxiliar_4} implies \eqref{integral_equation_critical_measure} for $h$.
\end{proof}

A combination of Lemmas \ref{lemma_test_function_at_infinity} and \ref{lemma_test_function_singularity}
leads to the following.
 
\begin{cor} \label{corollary_test_function}
Let $\mu$ be a $\varphi$-critical  measure. Suppose $z_1, z_2, \ldots, z_{N-1}$ are distinct
points such that
 \begin{equation}\label{equacao_auxiliar_27a}
 \int\frac{d\mu(x)}{|x-z_j|}<+\infty, \qquad j=1, \ldots, N-1.
 \end{equation}
Then equation 
\eqref{integral_equation_critical_measure} is valid for
the function
\[ h(x) = \prod_{j=1}^{N-1} \frac{1}{x-z_j}. \]
\end{cor}
\begin{proof}
Let $\delta = \frac{1}{4} \min \{ |z_j-z_k| \mid 1 \leq j < k \leq N-1 \}$.
For each $j=1, \ldots, N-1$, we can take a $C^2_c$ function $\chi_j : \mathbb C \to [0,1]$  
supported in $D(z_j, 2\delta)$ with $\chi_j \equiv 1$ in $D(z_j, \delta)$. 
Then $h_j = \chi_j h$ satisfies the condition of Lemma \ref{lemma_test_function_singularity}
and 
$h - \sum_{j=1}^{N-1} h_j$ satisfies the condition of Lemma \ref{lemma_test_function_at_infinity}.
Thus the equality \eqref{integral_equation_critical_measure} holds for these functions,
and then by linearity it also holds for $h$.
\end{proof}

Given a finite positive Borel measure $\mu$, consider the function $\C \to [0,+\infty]$
$$
z\mapsto \int \frac{d\mu(x)}{|x-z|}.
$$
It is the convolution of a finite Borel measure with the function $x\mapsto \frac{1}{|x|}$, 
which is in $L^1_{loc}(\C,m_2)$, where $m_2$ is the planar Lebesgue
measure. Thus Tonelli's Theorem tells us it also belongs to $L^1_{loc}(\C,m_2)$, being then finite $m_2$-a.e.
This implies that the Cauchy Transform $C^\mu$ of $\mu$,
$$
	C^\mu(z)=\int \frac{d\mu(x)}{x-z}, \qquad z\in\C,
$$
is well-defined and finite $m_2$-a.e., and
\begin{equation}\label{behavior_cauchy_transform_infinity}
C^\mu(z)=\Boh(z^{-1}), \quad z\to\infty.
\end{equation}

As observed by Martínez-Finkelshtein and Rakhmanov \cite[Lemma 5.1]{martinez_rakhmanov}, 
an essential feature of critical measures is the fact that their Cauchy Transform satisfies an
algebraic equation of degree $2$.

\begin{prop}\label{proposition_algebraic_equation_critical_measure}
Let $\mu$ be a $\varphi$-critical measure. Then there exists a polynomial $R$ of degree $2N-2$ such that
\begin{equation}\label{algebraic_equation_cauchy_transform}
\left( C^\mu(z)+\frac{1}{2}V'(z) \right)^2=R(z), \qquad m_2 \text{-a.e.}
\end{equation}
where $m_2$ is the Lebesgue measure on $\C$.
\end{prop}
% Proposition \ref{proposition_algebraic_equation_critical_measure} should be compared to \cite[Lemma~5.1]{martinez_rakhmanov}. 
As we will see later, one of the consequences of Proposition
\ref{proposition_algebraic_equation_critical_measure} is that $\varphi$-critical measures are 
always compactly supported, see Proposition \ref{proposition_support_critical_measure} below. 
However, since we cannot assume this {\it a priori},
our proof, at the technical level, is different from the one given in \cite[Proof of Lemma 5.1]{martinez_rakhmanov}, 
although the key ideas are similar.

\begin{proof}
Since $\int \frac{d\mu(x)}{|x-z|}$ is finite $m_2$-a.e., we can fix $N-2$ distinct points $z_1,\hdots,z_{N-2}$ for which
\begin{equation}\label{assumption_cauchy_transform}
\int \frac{d\mu(x)}{|x-z_j|}<\infty,\quad j=1,\hdots,N-2.
\end{equation}

Take $z\in \C \setminus \{z_1,\hdots,z_{N-2}\}$. Since we want to have 
\eqref{algebraic_equation_cauchy_transform} $m_2$-a.e.\ and $\int \frac{d\mu(x)}{|x-z|}$ is finite
$m_2$-a.e., it is enough to prove the identity \eqref{algebraic_equation_cauchy_transform} under the assumption
\begin{equation}\label{assumption_cauchy_transform_2}
\int \frac{d\mu(x)}{|x-z|}<\infty.
\end{equation}

Define
\begin{equation}\label{definition_A_function}
h(x)=\frac{A(x)}{x-z}, \qquad A(x)=\prod_{j=1}^{N-2}(x-z_j)^{-1}.
\end{equation}
Then $h$ satisfies the conditions of Corollary \ref{corollary_test_function} (with $z=z_{N-1}$)
and so \eqref{integral_equation_critical_measure} is valid for
this $h$. We can write
\begin{equation}\label{equacao_auxiliar_14}
\int V'(x)h(x)\, d\mu(x) = A(z)V'(z)C^\mu(z)+D_1(z),
\end{equation}
where $D_1$ is the rational function
$$
D_1(z)=\int \frac{A(x)V'(x)-A(z)V'(z)}{(x-z)}\, d\mu(x),
$$
whose only possible poles are the points $z_1,\hdots,z_{N-2}$, all of them simple, and
\begin{equation}\label{D_1_behavior_infinity}
D_1(z)=\Boh(1),\qquad z\to\infty.
\end{equation}

On the other hand, we can write
$$
\frac{h(x)-h(y)}{x-y}=\frac{(x-y)A(z)+(z-x)A(y)+(y-z)A(x)}{(x-y)(x-z)(y-z)} -\frac{A(z)}{(x-z)(y-z)}.
$$
The first term in the right-hand side is bounded away from $z_1,\hdots,z_{N-2}$, and since we 
are assuming \eqref{assumption_cauchy_transform} this term is $d\mu(x)\times
d\mu(y)$-integrable. The second term is also
{$d\mu(x)\times d\mu(y)$-integrable} because of assumptions \eqref{assumption_cauchy_transform} 
and \eqref{assumption_cauchy_transform_2}. From this we conclude
\begin{equation}\label{equacao_auxiliar_15}
\iint \frac{h(x)-h(y)}{x-y} \, d\mu(x)d\mu(y) = -A(z)(C^\mu(z))^2 + D_2(z),
\end{equation}
where
$$
D_2(z)=\iint \frac{(x-y)A(z)+(z-x)A(y)+(y-z)A(x)}{(x-y)(x-z)(y-z)} \, d\mu(x)d\mu(y),
$$
is also a rational function with simple poles at $z_1,\hdots,z_{N-2}$ and no other poles. Moreover,
\begin{equation}\label{D_2_behavior_infinity}
 D_2(z)=\Boh(z^{-1}),\quad z\to\infty.
\end{equation}

As \eqref{integral_equation_critical_measure} is valid for this $h$, 
equations \eqref{equacao_auxiliar_14} and \eqref{equacao_auxiliar_15} give us
$$
-A(z)(C^\mu(z))^2 + D_2(z)=A(z)V'(z)C^\mu(z)+D_1(z),
$$
which is equivalent to \eqref{algebraic_equation_cauchy_transform} if we set
\begin{equation} \label{expression_R}
R(z)=\frac{D_2(z)-D_1(z)}{A(z)}+\frac{1}{4}\left(V'(z)\right)^2.
\end{equation}
Note that the poles of $A$ cancel out the possible poles of $D_1-D_2$, and due also 
to the behavior of $D_1$ and $D_2$ at infinity given in
\eqref{D_1_behavior_infinity}, \eqref{D_2_behavior_infinity} we see that $R$ is indeed 
a polynomial of degree $2N-2$.
\end{proof}

From the polynomial $R$ in \eqref{algebraic_equation_cauchy_transform}, we can recover $\mu$.

\begin{prop}\label{proposition_density_critical_measure}
 Suppose $\mu$ is a measure on $\mathbb C$ for which there exist polynomials $Q$ and $R$ such that
\begin{equation}\label{algebraic_equation_cauchy_transform_2}
\left( C^\mu(z) + Q(z) \right)^2 = R(z), \qquad m_2 \text{-a.e.}
\end{equation}
Then $\mu$ is supported on a union of analytic arcs, which are maximal trajectories of 
the quadratic differential
$-R(z) \, dz^2$. Moreover, in the interior of any arc of $\supp\mu$, the measure $\mu$ 
is absolutely continuous with respect to the arclength measure, with density given by
\begin{equation}\label{absolutely_continuous_critical_measure}
	d\mu(s)=\frac{1}{\pi i}\sqrt{R(s)}\, ds,
\end{equation}
where $ds$ is the complex line element, chosen according to a fixed orientation of the 
arcs of $\supp \mu$.
\end{prop}

Observe that by Proposition \ref{proposition_algebraic_equation_critical_measure}, the condition \eqref{algebraic_equation_cauchy_transform_2} is valid  for $\varphi$-critical
measures with the choice $Q(z)=\frac{V'(z)}{2}$.

Since Proposition \ref{proposition_density_critical_measure} is essentially about local properties
of the measure $\mu$, it is not important if the measure $\mu$ has unbounded support or not. 
So the proof presented here is essentially the same as the one  given in \cite[Lemma~5.2]{martinez_rakhmanov} for
measures with bounded support, which in turn is modelled after a result in \cite{bergkvist_rullgard}.
We decided to give a full detailed proof here for the benefit of the reader, and also to correct a sign misprint
in \cite{martinez_rakhmanov}: in the first centered formula following formula (5.23) 
in \cite{martinez_rakhmanov} the right-hand side should have a $-$-sign.

\begin{proof}
Let $G$ be a simply connected domain not containing zeros of $R$ and intersecting $\supp\mu$. 
Fix a point $z_0\in \supp\mu\cap G$ and select a branch of $\sqrt R$ in
$G$, and define
$$
\xi(z)=\int_{z_0}^z\sqrt{R(s)} \, ds, \quad z\in G.
$$
Reducing $G$ if necessary, we can assume $\xi$ is a conformal mapping between $G$ and 
$\xi(G)=I\times iJ=\widehat G$, $I,J\subset \R$ intervals both containing $z=0$.

For this same branch of $\sqrt R$, define for $m_2$-a.e.\ $z\in G$ a function $\chi$ by the formula
$$
\chi(z)=\frac{C^\mu(z)+Q(z)}{\sqrt{R(z)}}.
$$
Due to \eqref{algebraic_equation_cauchy_transform_2} $\chi$ assumes values in $\{-1,1\}$. 
A simple application of the Cauchy-Green Formula \cite[pg.~491]{Greene_Krantz_book} gives us
$$
\frac{\partial C^\mu}{\partial \overline z} = -\pi  d\mu
$$
in the sense of distributions, and we conclude (also in the sense of distributions)
\begin{equation}\label{derivative_chi_z}
	\frac{\partial}{\partial \overline z} \, \left(\sqrt{R(z)}\chi(z)\right)=
	\sqrt{R(z)}\, \frac{\partial \chi}{\partial\overline z}(z)=-\pi d\mu(z).
\end{equation}

Now, $\xi$ is a conformal map from $G$ onto $\widehat G$, with inverse denoted by
$z=z(\xi)$, which induces a map between distributions in the $z$-variable and 
the $\xi$-variable, say $u\mapsto u_*$,
where $u_*$ is the distribution acting on test functions in $\widehat{G}$ via the formula
$$
\langle u,\phi\rangle = \langle u_*, \psi \rangle,
$$
where $\psi(\xi)=\phi(z)$.

Now we will calculate the (distributional) derivative of $\chi_*(\xi)=\chi(z(\xi))$. 
If $\mu_*$ is the pushforward measure of $\mu$
induced by $\xi$ in $\widehat G$ then \eqref{derivative_chi_z} and the chain rule applied to $\psi(\xi)=\phi(z)$ imply
\begin{align*}
\langle \mu_*,\psi \rangle & =  \langle \mu,\phi \rangle \\
			    & =  -\frac{1}{\pi}\left\langle \frac{\partial}{\partial \overline z}\left(\sqrt R \, \chi\right),\phi\right\rangle \\
			    & =  \frac{1}{\pi}\int \frac{\partial \phi}{\partial \overline z}(z) \sqrt{R(z)} \, \chi(z)\, dm_2(z) \\
			    & =  \frac{1}{\pi}\int \frac{\partial}{\partial \overline z} \left( \psi(\xi(z)) \right) \sqrt{R(z)} \, \chi(z) \, dm_2(z) \\
			    & =  \frac{1}{\pi}\int \left( \frac{\partial \psi}{\partial \xi}(\xi(z)) \frac{\partial\xi}{\partial \overline z}(z) 
				    + \frac{\partial \psi}{\partial\overline\xi}(\xi(z))\overline{\frac{\partial\xi}{\partial z}(z)} \right)\sqrt{R(z)}\chi(z) \, dm_2(z) \\
			    & =  \frac{1}{\pi}\int \frac{\partial \psi}{\partial\overline\xi}(\xi(z))\left| \sqrt{R(z)} \right|^2 \chi(z) \, dm_2(z),
\end{align*}
since $\frac{\partial \xi}{\partial \overline{z}} = 0$ and $\frac{\partial \xi}{\partial z} = \sqrt{R(z)}$.
Thus by the Change of Variables Formula  for the Lebesgue measure
\begin{align*}
\langle \mu_*, \psi \rangle	& =  \frac{1}{\pi} \int \frac{\partial \psi}{\partial\overline\xi}(\xi) 
								\left| \sqrt{R(z(\xi))} \right|^2 \chi_*(\xi) \, \left|z'(\xi)\right|^2\, dm_2(\xi) \\
			    & =  \frac{1}{\pi}\int \frac{\partial \psi}{\partial\overline\xi}(\xi) \, \chi_*(\xi)\, dm_2(\xi),
					\end{align*}
where we used that $z'(\xi)	= \frac{1}{\xi'(z)} = \frac{1}{\sqrt{R(z)}}$.
The result is that $ \langle \mu_*, \psi \rangle	 =	
- \frac{1}{\pi}\left\langle \frac{\partial \chi_*}{\partial\overline\xi},\psi\right\rangle $
which can be rewritten as
\begin{equation}\label{distributional_derivative_chi_xi}
\frac{\partial \chi_*}{\partial\overline \xi}=\frac{1}{2}\frac{\partial \chi_*}{\partial x}+\frac{i}{2}\frac{\partial
\chi_*}{\partial y}=-\pi d\mu_*,
\end{equation}
where $\xi=x+iy$ with $x,y\in\R$. As $\mu$ is a real measure, we conclude $\frac{\partial \chi_*}{\partial y}=0$, which means
$\chi_*(\xi)=g(\re \xi)$, for a real function $g$ defined on the interval $I$.

If $\lambda$ is the pushforward measure of $\mu_*$ induced by $\xi\mapsto \re \xi$, using
\eqref{distributional_derivative_chi_xi} we can conclude in a similar fashion that
$$
\frac{dg}{dx}=-\alpha d\lambda,
$$
where $\alpha$ is a positive constant, and so
$$
g(x)=\beta - \alpha \int_{-\infty}^x d\lambda,
$$
for some real constant $\beta$. But $g$ only assumes values in $\{1,-1\}$, and since $\int_{-\infty}^xd\lambda$ is
non-decreasing, we conclude $g$ is non-increasing and also that there exists $x_0$ in $I$ for which $\lambda=c\delta_{x_0}$,
where $c$ is a positive constant, and then
$$
g(x)=1-2\int_{-\infty}^x d\delta_{x_0}.
$$

In terms of $\mu_*$ and $\chi_*$, this means that there exists a vertical segment 
$L=\{x_0\}\times iJ$ in $\widehat G$ with $\chi_*$ equal to $1$ at the left side of $L$ and $-1$ at
the right side of $L$, and $\supp \mu_*$ is contained in $L$.
But for such $\chi_*$ a direct computation shows
$$
\frac{\partial \chi_*}{\partial \overline \xi}=\frac{1}{2}\frac{\partial \chi_*}{\partial x} =-dy,
$$
where $dy$ is the Lebesgue measure on $L$, so by \eqref{distributional_derivative_chi_xi}
$$
d\mu_*=\frac{1}{\pi}dy=\frac{1}{i\pi}dt.
$$
where $dt$ is the line element in $L$, oriented such that $\chi_*=1$ on the positive side of $L$. 
Since $\xi$ is a conformal map we can pullback the measures $d\mu_*$ and $dt$, obtaining
$$
d\mu(s)=\frac{1}{\pi i}\sqrt{R(s)}\, ds,
$$
where $ds$ is the complex line element on the trajectory $\gamma=\xi^{-1}(L)$ chosen
with orientation such that $\xi(z)=1$ on the positive side (i.e., on the left-hand side) of $\gamma$.

Finally, the trajectory arcs of $\supp\mu$ must be maximal, because the arguments above 
show that $\supp\mu$ is an analytic arc in a neighborhood of any point of its support that
is not a zero of $R$.
\end{proof}

For the next proposition, it is important that $\varphi=\re V$ with $V$ polynomial.

\begin{prop}\label{proposition_support_critical_measure}
For any $\varphi$-critical measure $\mu$, the support $\supp\mu$ is compact and can be written as
\begin{equation}\label{decomposition_support_critical_measure}
\supp\mu = \bigcup_j\alpha_j,
\end{equation}
where each $\alpha_j$ is an analytic arc and the union is taken over a finite set of indices.
Each $\alpha_j$ is a critical trajectory of the quadratic differential $-R(z) dz^2$.
\end{prop}

\begin{proof} 
Propositions \ref{proposition_algebraic_equation_critical_measure} and 
\ref{proposition_density_critical_measure} tell us that the measure $\mu$ is supported on maximal
trajectories of the quadratic differential 
$$
-R(z)dz^2=-\left(C^{\mu}(z)+\frac{1}{2}V'(z)\right)^2dz^2,
$$
which has a pole of order $2N+2$ at $z=\infty$. The general theory of quadratic differentials assures 
us that there exists a neighborhood $U$ of $z=\infty$ such that
every trajectory intersecting $U$ ends up at $z=\infty$ in at least one direction 
\cite[Theorem 7.4]{strebel_book}. In particular, if $\gamma\cap U\neq \emptyset$ for some maximal
trajectory $\gamma$ of $\supp \mu$, then $\gamma$ would be unbounded, and due to 
\eqref{absolutely_continuous_critical_measure} the measure $\mu$ would not be finite.

Since $R$ is a polynomial, the only bounded trajectories are critical trajectories
that connect two different zeros of $R$. Since there are only finitely many zeros of $R$,
the union in \eqref{decomposition_support_critical_measure} is
certainly over a finite number of arcs.
\end{proof}

It is worth noting that the  expression \eqref{expression_R} for $R$ in terms of the function $A$ defined in
\eqref{definition_A_function}, depends on the chosen points $z_1,\hdots,z_{N-2}$ at which the 
Cauchy transform of $\mu$ is convergent. Since we now know by Proposition \ref{proposition_support_critical_measure}
that $\supp\mu$ is compact, we can take the constant function $A(x) = 1$ in \eqref{definition_A_function},
and the rest of the proof of \eqref{algebraic_equation_cauchy_transform} then works fine
for this $A$. % Then we remove the aparent dependence of $R$ to these points $z_1,\hdots,z_{N-2}$. 
Repeating all the computation, we also end up with a nicer representation for $R$,
\begin{equation}\label{alternative_expression_R}
R(z)= \left(\frac{V'(z)}{2} \right)^2 -\int \frac{V'(x)-V'(z)}{x-z}\, d\mu(x),
\end{equation}
which is well-known for equilibrium problems in polynomial external fields on the real line, see e.g.\ 
\cite{deift_kriecherbauer_mclaughlin}. This representation also allows us to derive some
extra properties for $R$. For example, from \eqref{alternative_expression_R} one easily sees
that $R(z)-\frac{(V'(z))^2}{4}$ is a polynomial of degree $n-2$, whose leading 
coefficient coincides with minus the leading coefficient of $V'$.

Also, if $V$ is quadratic, say
$$
V(z)=az^2+bz+c, \qquad a \neq 0,
$$
then \eqref{alternative_expression_R} reduces to
$$
R(z)=\left(\frac{2az+b}{2}\right)^2-2a,
$$
since $\mu$ is a probability measure.
Then the zeros of $R$ are 
$$
z_\pm=\frac{-b\pm 2\sqrt{2a}}{2a},
$$
for some  choice of the branch of the square root, and the support of 
the  $\varphi$-critical measure is the segment from $z_-$ to $z_+$. 

For $V$ cubic, the polynomial $R(z)-\frac{(V'(z))^2}{4}$ is linear, so there is one free 
parameter to be determined by the requirement of extra conditions, see for instance
\cite{alvarez_alonso_medina_s_curves,huybrechs_lejon_kuijlaars,deano_huybrechs_kuijlaars}.
See also \cite{deano_orthogonal_polynomials_bounded_interval} for a similar situation.

\subsection{Critical measures and the S-property}

By Proposition \ref{proposition_density_critical_measure}, the support of a $\varphi$-critical 
measure $\mu$ consists of analytic arcs and their endpoints. The analytic arcs are trajectories
of the quadratic differential $-R(z) dz^2$. Each regular point $z \in \supp \mu$ has a 
neighborhood, say $G$, for which $G\cap\supp\mu$ is an analytic arc. 
The only non-regular points are the zeros of $R$.

The next result is taken from \cite{martinez_rakhmanov}.

\begin{prop}\label{proposition_s_property_critical_measures}
Let $\mu$ be a $\varphi$-critical measure, with $\supp \mu=\bigcup \alpha_j$
as in Proposition \ref{proposition_support_critical_measure}.
\begin{enumerate}[i)]
\item Then there exist constants $\omega_j \in \R$ such that
$$
U^\mu(z)+\frac{1}{2}\varphi(z)=\omega_j, \qquad z\in\alpha_j.
$$
\item If $z\in\supp\mu$ is regular, then
\begin{equation*}
\frac{\partial }{\partial n_+}\left(U^\mu+\frac{1}{2}\varphi\right)(z)=\frac{\partial}{\partial
n_-}\left(U^\mu+\frac{1}{2}\varphi\right)(z),
\end{equation*}
where $n_\pm$ are the normal vectors to $\supp\mu$ at $z$ pointing in opposite directions.

\end{enumerate}
\end{prop}

\begin{proof} 
See Lemma 5.4 of \cite{martinez_rakhmanov}.
\end{proof}

An immediate consequence of part  ii) of Proposition 
\ref{proposition_s_property_critical_measures} is.
\begin{cor}\label{corollary_s_property}
If $V$ is a polynomial, $\varphi=\re V$ and $\Gamma\subset \C$ is a contour satisfying the growth condition \eqref{growth_condition} and whose equilibrium 
measure $\mu^{\varphi,\Gamma}$ is a $\varphi$-critical measure, then $\Gamma$ has the
S-property in the external field $\varphi$.
\end{cor}

As a final remark, we notice that the cornerstone of the present section is the Proposition 
\ref{proposition_algebraic_equation_critical_measure}, which gives us that the Cauchy
transform of a $\varphi$-critical measure is an algebraic function. One could also analyse 
\eqref{algebraic_equation_cauchy_transform} from another perspective, namely given $V'$,
find a polynomial $R$ and a probability measure $\mu$ such that the algebraic equation 
\eqref{algebraic_equation_cauchy_transform} is satisfied. This is the spirit of some works
already present in the literature, e.g.  \cite{shapiro_holst,shapiro_takemura_tater,shapiro_tater,martinez_rakhmanov}. 
It is expected that the pairs $(R,\mu)$ satisfying
\eqref{algebraic_equation_cauchy_transform} depend on continuous parameters, but as it follows 
from the present work, only for a finite number of pairs $(R,\mu)$ the measure $\mu$
is an equilibrium measure (of some admissible contour) in the external field $\varphi=\re V$. More precisely, 
just for a finite number of pairs $(R,\mu)$ obtained in this way,
all the constants $\omega_j$ appearing in Proposition \ref{proposition_s_property_critical_measures} i) coincide. 
In such a case  there is a choice of partition
$\mathcal P$ and a contour $\mathcal T(\mathcal P)$ for which $\mu$ is the equilibrium measure of $\Gamma$
in the external field $\varphi$.

%%%%%%%%%%%%%%%%%%%%%%%%%%%%%%%%%%%%%%%%%%%%%%%%%%%%%%%%%%%%%%%%%%%%%%
\section{Critical sets}\label{critical_measures_critical_sets}
%%%%%%%%%%%%%%%%%%%%%%%%%%%%%%%%%%%%%%%%%%%%%%%%%%%%%%%%%%%%%%%%%%%%%%

We continue to use the notation introduced in the beginning of Section \ref{critical_measures}.
So for $h \in C^2_c$ we have $h_t$ as in \eqref{h_variation}, and we
denote by $\mu^t$ the pushforward of a measure $\mu$ induced by 
$h_t$. For a subset $F\subset \C$, we denote $F^t=h_t(F)$.

In Section \ref{critical_measures} we studied variations of the energy 
functional $I^\varphi$ induced by $h$, viewed as acting
on measures, see \eqref{equation_derivative_energy_measure}. 
But we might as well study variations of $I^{\varphi}$ when viewed as 
acting on sets, that is, we might consider the limit
\begin{equation}\label{equation_derivative_energy_compact}
D_hI^\varphi(F)=\lim_{t\to 0}\frac{I^\varphi(F^t)-I^\varphi(F)}{t}.
\end{equation}

Just as the vanishing of $D_h I^{\varphi}(\mu)$ defines critical measures,
the vanishing of $D_h I^{\varphi}(F)$ defines critical sets.

\begin{definition} \label{definition_critical_set}
We say that a closed set $F$ is a $\varphi$-critical set (or simply a critical set) if
the limit in \eqref{equation_derivative_energy_compact} exists for every $h \in C_c^2$, and
\[	D_h I^{\varphi}(F) = 0 \]
for every $h \in C_c^2$.
\end{definition}

The remarkable fact is that the limit \eqref{equation_derivative_energy_compact} 
coincides with the derivative of $I^\varphi$ at the equilibrium measure in external field of $F$.
\begin{prop}\label{proposition_derivative_energy_compact} 
Let  $F\subset \C$ be a closed set with $-\infty<I^\varphi(F)<+\infty$. Then
the limit in \eqref{equation_derivative_energy_compact} exists and 
\begin{equation}
 D_hI^\varphi(F) = D_hI^\varphi(\mu^{\varphi,F}).
 \end{equation}
\end{prop}

The existence of the equilibrium measure in the external field $\mu^{\varphi,F}$ is guaranteed by 
Theorem \ref{rakhmanov_theorem_1}. Proposition \ref{proposition_derivative_energy_compact} is due
to Rakhmanov, who gives a proof in \cite[Section 9.10]{rakhmanov_orthogonal_s_curves} and refers
to an unpublished manuscript where the argument is first given. 
We found that one of the steps in the proof in \cite{rakhmanov_orthogonal_s_curves}, 
namely the equivalent of \eqref{limite_derivada_mu} below, might require some additional explanation, and
therefore we decided to present here a detailed proof.

To prove Proposition \ref{proposition_derivative_energy_compact}, we will need an auxiliary lemma. 
We use $\mathcal M_0$ to denote the set of signed measures defined on $\C$, whose
positive and negative parts belong to $\mathcal M_1(\mathcal \C)$ (see the beginning of
Section \ref{potential_theory_notions} for the definition of $\mathcal M_1(\mathcal \C)$). 
Recall that a sequence of finite signed measures $(\sigma_n)$
converges vaguely to a finite signed measure $\sigma$ if
$$
\int f d\sigma_n\stackrel{n\to\infty}{\to} \int f d\sigma,
$$
for every continuous functions $f$ with compact support.

\begin{lem}\label{lemma_vague_convergence}
 The functional
$$
\mathcal M_0 \ni \sigma \mapsto I(\sigma)=\iint \log|x-y|^{-1}d\sigma(x)d\sigma(y)
$$
is well-defined on $\mathcal M_0$ and strictly positive for $\sigma \neq 0$. 
If $(\sigma_n)$ is a sequence of signed measures in $\mathcal M_0$ with 
$$
I(\sigma_n)\stackrel{n\to\infty}{\to} 0,
$$
then the sequence $(\sigma_n)$ converges vaguely to the null measure.
\end{lem}
\begin{proof}
The lemma is well known if we consider signed measures supported in the unit disc, 
see for example \cite[pg.~80, Theorem~1.16, pg.~88, Lemma~1.3]{Landkof_book}. The main issue
here is that we are dealing with signed measures  with possibly unbounded support. 
The proof we are going to give is actually obtained by collecting known results together and fitting them to
our needs here.

Given two measures $\mu,\nu$ with finite energies $I(\mu)$, $I(\nu)$, the integral
$$
I(\mu,\nu)=\iint \log\frac{1}{|x-y|}d\mu(x)d\nu(y),
$$
is well defined and finite \cite[Theorem 2.2]{mattner_singular_integrals}, that is, 
$|\log|x-y|^{-1}|\in L^1(\mu\times\nu)$. In particular, if
$\sigma=\sigma_+-\sigma_-\in \mathcal M_0$, then the value
$$
I(\sigma)=I(\sigma_+)+I(\sigma_-)-2I(\sigma_+,\sigma_-)
$$
is also well defined and finite. Due to a result of Mattner \cite[Theorem 2.2]{mattner_singular_integrals}
it is positive, i.e., $I(\sigma) \geq 0$ and $I(\sigma) = 0$ if and only if $\sigma$ is the zero measure.

For $\sigma\in\mathcal M_0$ compactly supported and absolutely continuous with respect to 
the Lebesgue measure $m_2$, say $d\sigma=fdm_2$, with $f\in C^\infty_c$ and
real-valued satisfying $\int fdm_2=0$, Cegrell et al.
\cite[Lemma 2.4]{cegrell_kolodziej_levenberg} obtained the representation 
\begin{equation}\label{expression_energy_absolutely_continuous_measure}
I(\sigma)=2\pi\int \frac{|\hat f(x)|^2}{|x|^2}\, dm_2(x),
\end{equation}
for $I(\sigma)$, where
$$
\hat f(z)=\int f(x)e^{-2\pi i \re(x\overline z)}\, dm_2(x)
$$
is the Fourier transform of $f$ (induced by $\R^2$, so that $\re(x\overline z)=\re x\re z+\im x \im z$ is just 
the inner product in $\R^2$). For
$\sigma=\sigma_+-\sigma_-\in\mathcal M_0$ and a given sequence of approximation to the identity $(\psi_n)$ 
($\psi_n$ is smooth and independent of $\sigma$), satisfying 
$$
\int \psi_n \, dm_2=1, \quad \psi_n(z)\geq 0, \ \text{ for } z\in \C, \text{ and } \quad
 \psi_n(z) \equiv 0 \ \text{ for }  |z|>\frac{1}{n},
$$
Cegrell et al.\ also constructed a sequence of absolutely continuous 
signed measures $\mu_n=\psi_n\ast \lambda_n \in \mathcal M_0$ with
$
\lim_{n \to \infty} I(\mu_n)  =  I(\sigma),
$
where 
\begin{equation}\label{definition_sequence_convolution}
\lambda_n=\restr{\sigma}{\Sigma_n}-\sigma(\Sigma_n)\lambda,
\end{equation}
and $\lambda$ is the normalized Lebesgue measure on the unit disc and $(\Sigma_n)$ a 
sequence of discs $\Sigma_n=D_{R_n}$ whose respective sequence of radii $(R_n)$ converges to
$+\infty$ when $n\to \infty$.

So now consider a sequence $(\sigma_m)$ in $\mathcal M_0$ with $I(\sigma_m)\to 0$.
For each $m$ consider the measures $\lambda_{n,m}$ defined as in
\eqref{definition_sequence_convolution} for $\sigma=\sigma_m$. For each $m$, 
we can take $n=n(m)$ large enough, such that for $\mu_m :=\lambda_{n,m} \ast \psi_n$,
\begin{equation}\label{limit_energies_approximation}
|I(\mu_m)-I(\sigma_m)|<\frac{1}{m},
\end{equation}
and we can make sure that $n(m+1)>n(m)$, for all $m$. The explicit form of 
$\mu_m$ and this last condition on the indices imply in particular 
\begin{equation}\label{vague_convergence_convolution}
\mu_m-\sigma_m\to 0 \text{ vaguely.}
\end{equation}

The measures $\mu_m$ are absolutely continuous w.r.t.\ $m_2$ with a smooth and compactly 
supported density, say $d\mu_m=f_mdm_2$. Equations
\eqref{expression_energy_absolutely_continuous_measure} and \eqref{limit_energies_approximation} 
and the convergence $I(\sigma_m)\to 0$ then imply
$$
\frac{\hat f_m(x)}{x}\to 0 \, \text{ in } L^2(m_2) \qquad \text{ as } m \to \infty.
$$

If $h=h_1+ih_2\in C^\infty_c$, then $h_j,f_m\in L^1(m_2)\cap L^2(m_2)$ and 
also $x\hat h_j\in L^2(m_2)$, and the Plancherel Theorem and Cauchy-Schwarz inequality imply for $j=1,2$,
\begin{align*}
\left|\int h_j\, d\mu_m\right|  & =   \left|\int h_j(x)f_m(x)\, dm_2(x)\right| \\
				 & = \left| \int \hat h_j(z)\hat f_m(z)\, dm_2(z) \right| \\
				 & \leq \|x\hat h_j\|_2 \left\|\frac{\hat f_m}{x}\right\|_2 \to 0,
					\qquad \text{ as } m \to \infty,
\end{align*}
so that 
\begin{equation}\label{limit_integral_test_functions}
	\lim_{m \to \infty} \int h \; d\mu_m  =  0.
\end{equation}

Finally, since $C^\infty_c$ is dense in $C_c$ w.r.t. the uniform norm and $|\int f_m\, dm_2|$ is 
uniformly bounded, the inequality 
$$
\left|\int h \, d\mu \right|\leq \|h-h_n\|_{\infty}\left|\int f_m\, dm_2\right|+\left|\int h_n\, d\mu_m\right|
$$
for a sequence $(h_n)$ of functions in $C^\infty_c$ converging to $h\in C_c$ in the uniform norm 
shows the limit
\eqref{limit_integral_test_functions} is also true for $h\in C_c$, that is, $\mu_m\to 0$ vaguely. 
In virtue of \eqref{vague_convergence_convolution} this is equivalent to the
vague convergence $\sigma_m\to 0 $.
\end{proof}

\begin{proof}[Proof of Proposition \ref{proposition_derivative_energy_compact}]
For $t$ small enough and $h_t$ as in \eqref{h_variation}, the inverse map $h_t^{-1}$ is a well defined $C^2_c$
function which satisfies
\begin{align}
 h^{-1}_t(x)                        & = x-th(x)+\boh(t), \label{first_equation_asympt_h_t}\\ 
\frac{h^{-1}_t(x)-h^{-1}_t(y)}{x-y} & = 1-t\frac{h(x)-h(y)}{x-y}+\boh(t),\label{second_equation_asympt_h_t}
\end{align}
with the $o$ terms uniform in $x,y \in \mathbb C$.

For simplicity, denote by $\mu_F$ and $\mu_{F^t}$ the equilibrium measures in the external field $\varphi$ of $F$ and $F^t$, 
respectively. By Rakhmanov's Theorem \ref{rakhmanov_theorem_1} and
the assumption $h\in C^2_c$ these measures certainly exist.

Clearly, in view of the definitions \eqref{equation_derivative_energy_measure} and 
\eqref{equation_derivative_energy_compact} and the Proposition \ref{proposition_formula_derivative},  
it is enough to prove that
$$
I^\varphi(F^t)-I^\varphi(F)=I^\varphi(\mu^t_F)-I^\varphi(\mu_F)+\boh(t), \quad  t\to 0,
$$
which comes down to proving
$$
I^\varphi(\mu_{F^t})-I^\varphi(\mu_F^t)=\boh(t), \quad t\to 0.
$$

For a given $t$, denote by $\mu_{F^t}^{-t}$ the pushforward measure of $\mu_{F^t}$ induced by the inverse mapping $h^{-1}_t$.
Bearing in mind equations \eqref{first_equation_asympt_h_t}, \eqref{second_equation_asympt_h_t} and mimicking the proof of
Proposition \ref{proposition_formula_derivative} we obtain
$$
I^\varphi(\mu_{F^t}^{-t})-I^\varphi(\mu_{F^t})=-tD_hI^\varphi(\mu_{F^t})+\boh(t), \quad  t\to 0,
$$
and also from Proposition \ref{proposition_formula_derivative},
$$
I^\varphi(\mu_F^t)-I^\varphi(\mu_F)=tD_hI^\varphi(\mu_F)+\boh(t), \quad  t\to 0.
$$
This in turn implies
\begin{align}\nonumber
 0 & \leq I^\varphi(\mu_F^t)-I^\varphi(\mu_{F^t}) \\ \nonumber
	 &   = I^\varphi(\mu_F)-I^\varphi(\mu_{F^t}^{-t})+t(D_hI^\varphi(\mu_F)-D_hI^\varphi(\mu_{F^t}))+\boh(t) \\
				           & \leq  t(D_hI^\varphi(\mu_F)-D_hI^\varphi(\mu_{F^t})) +\boh(t), \quad t \to 0, \label{equacao_auxiliar_5}
\end{align}
so we are done if we can show
\begin{equation}\label{limite_derivada_mu}
 \lim_{t \to 0} D_hI^\varphi(\mu_{F^t})  = D_hI^\varphi(\mu_F).
\end{equation}

To obtain \eqref{limite_derivada_mu}, assume for a moment
\begin{equation}\label{limite_fraco_mu}
 \int f \, d\mu_{F^t}\stackrel{t\to 0}{\to}\int f\, d\mu_F, \quad f\in C_c,
\end{equation}
that is, $\mu_{F^t}\to \mu_F$ vaguely, to be proved later.
Since $h\in C^2_c$ the function
$$
(x,y)\mapsto \frac{h(x)-h(y)}{x-y}
$$
is continuous with compact support, so that by \eqref{limite_fraco_mu}
\begin{equation}\label{integral_perturbacao_mu_1}
\iint\frac{h(x)-h(y)}{x-y}\, d\mu_{F^t}(x)d\mu_{F^t}(y)\stackrel{t\to 0}{\to} \iint\frac{h(x)-h(y)}{x-y}\, d\mu_{F}(x)d\mu_{F}(y).
\end{equation}
The function $hV'$ belongs also to $C^2_c$, so again by \eqref{limite_fraco_mu}
\begin{equation}\label{integral_perturbacao_mu_2}
 \int h(x)V'(x)\, d\mu_{F^t}(x)\stackrel{t\to 0}{\to} \int h(x)V'(x)\, d\mu_F(x).
\end{equation}
Using Proposition \ref{proposition_formula_derivative} and \eqref{integral_perturbacao_mu_1}, 
\eqref{integral_perturbacao_mu_2}, the limit \eqref{limite_derivada_mu} follows and
the proof is done, provided that we have \eqref{limite_fraco_mu}.

The proof of \eqref{limite_fraco_mu} will be done in two steps, namely $\mu_F^t \to \mu_F$ 
and $\mu_{F^t}-\mu_F^t \to 0$ vaguely. 

The first of these limits follows directly from the Change of Variables Formula,
$$
\int f\, d\mu_{F^t} = \int f\circ h_t \, d\mu \to \int f\, d\mu, \qquad t \to 0,
$$
because $f$ has compact support, and so it is bounded, and $h_t(z)\to z$ as $t\to 0$, and the 
Dominated Convergence Theorem can be applied.

The difference $D_hI^\varphi(\mu_F)-D_hI^\varphi(\mu_{F^t})$ is bounded, thanks to 
Proposition \ref{proposition_formula_derivative}, so \eqref{equacao_auxiliar_5} implies
$I^\varphi(\mu^t_F)-I^\varphi(\mu_{F^t})\to 0$. Now, following notations of the previous Lemma,
 \begin{align}\nonumber
0\leq I(\mu^t_F-\mu_{F^t})	 &  =   
				2I^\varphi( \mu^t_F) + 2I^\varphi( \mu_{F^t}) - 
					4I^\varphi \left( \frac{\mu^t_F+\mu_{F^t}}{2} \right) \\ \nonumber
					    &  \leq  2I^\varphi(\mu^t_F)+2I^\varphi(\mu_{F^t}) - 4I^\varphi(\mu_{F^t}) \\ 
					    &  =     2I^\varphi(\mu^t_F)-2I^\varphi(\mu_{F^t})\to 0  \label{equacao_auxiliar_6}
\end{align}
as $t \to 0$. 
The inequality in \eqref{equacao_auxiliar_6} is obtained by 
noting that $\frac{1}{2} ( \mu^t_F+\mu_{F^t}) \in \mathcal M^\varphi_1(F^t)$, so its weighted energy 
is larger than $I^\varphi(\mu_{F^t})$. Lemma \ref{lemma_vague_convergence} then implies $\mu^t_F-\mu_{F^t}\to 0$ vaguely.
\end{proof}	

An immediate consequence of Proposition \ref{proposition_derivative_energy_compact}  is
the following.

\begin{cor} \label{cor_critical_set}
The equilibrium measure $\mu^{\varphi,F}$ in the external field $\varphi$ of a $\varphi$-critical
set $F$ is a $\varphi$-critical measure.
\end{cor}

%%%%%%%%%%%%%%%%%%%%%%%%%%%%%%%%%%%%%%%%%%%%%%%%%%%%%%%%%%%%%%%%%%%%%%
\section{Proof of Theorem \ref{main_result}}\label{proof_main_theorem}
%%%%%%%%%%%%%%%%%%%%%%%%%%%%%%%%%%%%%%%%%%%%%%%%%%%%%%%%%%%%%%%%%%%%%%

In this section we are going to prove Theorem \ref{main_result}. 
Recall we are assuming that $V$ is a polynomial of degree $N$, $\varphi=\re V$ and $\mathcal T=\mathcal
T(\mathcal P)$ is the class of admissible contours associated to a fixed noncrossing partition 
$\mathcal P$, as given in Definition \ref{definition_admissible_sets}. 
This setup defines the max-min problem $(V,\mathcal T)$, see \eqref{max-min-problem}.

\subsection{The collection $\mathcal T_M$}

In the first step we are going to restrict the class of contours $\mathcal T$. The contours 
in $\mathcal T$
stretch out to infinity in certain directions in which $\varphi \to +\infty$, as specified by the
partition $\mathcal P$. However, it is not forbidden that some parts of $\Gamma \in \mathcal T$ are in regions
of the plane where $\varphi$ is very negative.  Then also $I^{\varphi}(\Gamma)$ will be very negative, and so
such $\Gamma$ will be far from optimal for the max-min problem \eqref{max-min-problem}. 

Here we are going to make this precise, and we show that for $M$ large enough we
can restrict to a certain subclass $\mathcal T_M$ of $\mathcal T$ that we are
going to define first.

For $M > 0$ large enough, the level set
\begin{equation}\label{level_sets}
	\varphi^{-1}(-M) = \{ z\in\C \mid \varphi(z)=-M\}.
\end{equation}
consists of $N$ disjoint analytic arcs, each of them stretching out to infinity in its both
ends. This is so because $\varphi$ is the real part of a polynomial of degree $N$. 
Then, if $M > 0$ is large enough,  
\begin{equation}\label{cut_off_set}
	\Delta_{M} = \varphi^{-1}(-\infty,-M) \cap \{ z\in \C \mid \dist(z, \varphi^{-1}(-M))>8\}
\end{equation}
is an open non-empty set and its boundary consists of a union of $N$ pairwise disjoint analytic arcs.
The distance in \eqref{cut_off_set} is the usual Euclidean distance
\[ \dist(z, X) = \inf_{x \in X}  |z-x|. \]
The set $\Delta_M$ is indicated by the gray regions in Figures~\ref{table_configurations} and 
\ref{level_curves}. We also assume $M$ satisfies 
\begin{equation}\label{condition_constant_M}
-M< \sup_{\Gamma\in \mathcal T}I^\varphi(\Gamma).
\end{equation}

Then for such $M$ we consider the subclass
\begin{equation}\label{definition_subclass_cut_off}
\mathcal T_M=\{\Gamma\in \mathcal T \mid \Gamma\cap \Delta_M=\emptyset\},
\end{equation}
and its closure 
\begin{equation} \label{definition_FM}
	\mathcal F_M = \overline{\mathcal T_M} 
	\end{equation}
with respect to Hausdorff metric on closed subsets of $\C$ with the hyperbolic distance.
The sets in $\mathcal F_M$ are closed subsets of $\C$, but they are not necessarily 
finite unions of contours, since this property is not preserved under taking closure in
the Hausdorff metric.

However, because of Definition \ref{definition_admissible_sets} ii),
each $F \in \mathcal F_M$ has at most $|\mathcal P_0|$ components that are all
unbounded in $\C$, since this property is preserved by taking closure in the Hausdorff metric.
Also
\begin{equation} \label{DeltaM_intersection_empty} 
F \cap \Delta_M  = \emptyset, \qquad \text{for all } F \in \mathcal F_M. 
\end{equation}

\begin{prop}\label{proposicao_fundamental}
%Every set $F\in \mathcal F_M$ is connected in $\overline \C$. Moreover,
We have
\begin{equation}\label{inequalities_energy_classes}
 \sup_{\Gamma\in \mathcal T \setminus \mathcal T_M} I^\varphi(\Gamma) <
	\sup_{\Gamma\in \mathcal T_M}I^\varphi(\Gamma)\leq 
 \sup_{F \in \mathcal F_M }I^\varphi(F) <+\infty.
\end{equation}
\end{prop}

\begin{proof}
% Connectivity in $\overline \C$ is preserved by Hausdorff convergence, so every 
% element in $\mathcal F_M$ is connected in $\overline \C$ because this is true for sets in $\mathcal T_M$.
The inequality in \eqref{inequalities_energy_classes} between the supremum 
over  $\mathcal T_M$ and the supremum over $\mathcal F_M$ is trivial.
\medskip

We start with the proof of the first inequality in \eqref{inequalities_energy_classes}. 
To prove this we take 
$\Gamma\in\mathcal T\setminus \mathcal T_M$. Since all connected components of $\Gamma$ are
unbounded, see item ii) of Definition \ref{definition_admissible_sets}, 
and $\Gamma$ intersects the set $\Delta_M$ from \eqref{cut_off_set}, we have that
$\Gamma$ contains a connected subset $\Gamma_0$ satisfying
$$
\Gamma_0 \subset \varphi^{-1}(-\infty,-M)\quad \text{and} \quad \diam \Gamma_0 \geq 8.
$$ 
Then $\capac(\Gamma_0) \geq \frac{\diam \Gamma}{4} \geq 2$, which means that $I(\Gamma_0) \leq \log \frac{1}{2}$, see 
for example \cite[page~138]{ransford_book}. Let $\omega$ be the equilibrium measure on $\Gamma_0$
(without external field).  Then $I(\omega) = I(\Gamma_0) \leq \log \frac{1}{2}$, and
$$ I^{\varphi}(\Gamma_0)  \leq I^{\varphi}(\omega) = I(\omega)  + \int \varphi \, d\omega \\
		\leq \log \frac{1}{2} - M,
$$
because $\Gamma_0$ is contained in the region where $\varphi < -M$.
Then by \eqref{condition_constant_M}
$$
I^\varphi(\Gamma)\leq I^\varphi(\Gamma_0) \leq \log\frac{1}{2}+\sup_{\Gamma' \in\mathcal T}I^{\varphi}(\Gamma').
$$

Since $\Gamma \in \mathcal T \setminus \mathcal T_M$ is arbitrary, and $\log \frac{1}{2} < 0$, we find 
the first inequality in \eqref{inequalities_energy_classes}.

\medskip

It remains to prove that $\sup_{F \in \mathcal F_M }I^\varphi(F)$ is finite.
To this end, we construct a closed curve $\gamma$, contained in $\C\setminus \Delta_M$ 
and intersecting all the connected components of $\partial \Delta_M$, see Figure \ref{level_curves} for a 
pictorial configuration of $\gamma$.
%\begin{center}
\begin{figure}
\begin{overpic}[scale=.6]{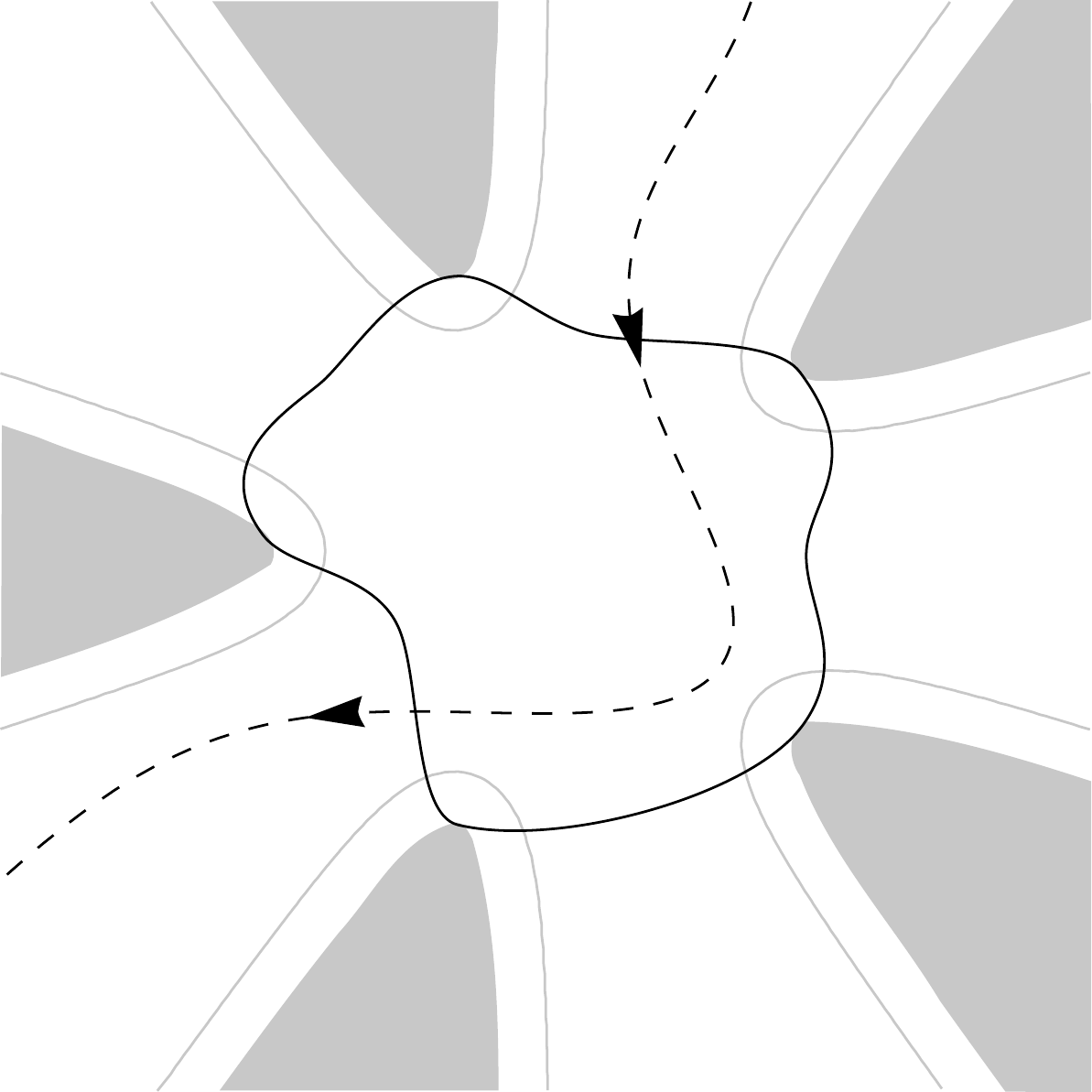}
  \put(60,126){$\gamma$}
  \put(13,45){$\Gamma_A$}
\end{overpic}
\caption{A possible choice for $\gamma$ in the proof of Proposition \ref{proposicao_fundamental} for the case $V(z)=z^5+z$.
The level curves $\varphi(z)=-M$ are shown in gray solid lines and the set $\Delta_M$
is the shaded region. Sets in $\mathcal F_M$ cannot intersect the shaded
region. The contour $\Gamma_A$, shown in black dashed line, is part of an admissible curve $\Gamma$ 
for $A=\{2,4\}$.  }\label{level_curves}
\end{figure}
%\end{center}

Recall that by item ii) of Definition \ref{definition_admissible_sets} 
the connected components of sets in $\mathcal T$ are unbounded and stretch out to infinity in at least two of
the sectors $S_1,\hdots,S_N$. Assuming in addition that $\Gamma\in \mathcal T_M$, 
that is, $\Gamma\cap \Delta_M = \emptyset$, it is easy to see that
$ \Gamma_A \cap \gamma\neq \emptyset$
for any connected component $\Gamma_A$ of $\Gamma$. This property is preserved 
by taking closures with respect to Hausdorff metric, that is, it is also
valid
\begin{equation}\label{equacao_auxiliar_21}
	F_A \cap\gamma\neq \emptyset
\end{equation}
for any connected component $F_A$ of a set $F$ belonging to $\mathcal F_M$. 

Fix $R>4$ for which $\gamma\subset D_{R-4}$. A similar reasoning as the one that led to \eqref{equacao_auxiliar_21} 
shows that 
\begin{equation}\label{equacao_auxiliar_22}
	F_A \cap \partial D_{R}\neq \emptyset,
\end{equation}
for any connected component $F_A$ of $F\in \mathcal F_M$.
From \eqref{equacao_auxiliar_21} and \eqref{equacao_auxiliar_22} we see that 
there exists a connected compact $F_0\subset F$ with
$$
F_0 \cap \gamma \neq \emptyset, \quad F_0 \cap \partial D_R \neq \emptyset, \quad
F_0\subset \overline D_{R}.
$$
This means that $\diam F_0\geq 4$, because $\gamma\subset D_{R-4}$. Then $\capac F_0 \geq 1$, which
implies $I(F_0) \leq 0$, and then
\begin{align*}
I^\varphi(F)  \leq I^\varphi (F_0) \leq I(F_0) + \sup_{z\in F_0} \varphi(z) 
	      \leq \sup_{z\in \overline D_{R}} \varphi(z) < +\infty.
\end{align*}
Since $R$ is independent of $F \in \mathcal F_M$, the 
inequality $\sup_{F \in \mathcal F_M }I^\varphi(F) <+\infty$ follows.
\end{proof}

\subsection{A maximizer set $F_0$}

The following is an immediate consequence of Proposition \ref{proposicao_fundamental}.

\begin{cor}\label{corollary_maximizing_set}
 There exists $F_0\in \mathcal F_M$ such that
\begin{equation}\label{equacao_auxiliar_20}
I^\varphi(F_0)=\sup_{F\in \mathcal F_M} I^\varphi(F).
\end{equation}
\end{cor}

\begin{proof}
The class $\mathcal F_M$ is closed in the Hausdorff metric. The upper semicontinuity 
of $I^\varphi$ given by Theorem \ref{theorem_upper_semicontinuity} and Proposition
\ref{proposicao_fundamental} give the desired result.
\end{proof}

The next lemma assures that small variations of $F_0$ are still in $\mathcal F_M$.

\begin{lem}\label{lemma_small_variations}
Let $F_0 \in \mathcal F_M$ be a maximizer set given by Corollary \ref{corollary_maximizing_set}. 
For every $h\in C_c^2$, there exists $t_0 > 0$ such that $F_0^t\in\mathcal F_M$, for 
every $t \in (-t_0, t_0)$.
\end{lem}

\begin{proof}
Since $F_0 \in \mathcal F_M$, we already observed that $F_0 \cap \Delta_M = \emptyset$, see
\eqref{DeltaM_intersection_empty}. 
The same argument that was used to prove the first inequality in Proposition 
\ref{proposicao_fundamental} shows that
\[
	F_0 \cap \overline \Delta_M = \emptyset. 
	\]

Let $h \in C_c^2$ be a test function. Since $h$ is compactly supported, it
follows that there is a $t_0 > 0$ such that
\begin{equation} \label{equacao_auxiliar_11a} 
	F_0^t \cap \overline{\Delta}_M = \emptyset \qquad \text{for all } t \in (-t_0, t_0). 
	\end{equation}
We also take $t_0$ small enough so that $x \mapsto h_t(x) = x + th(x)$ is a $C^2$-diffeomorphism 
on $\mathbb C$ for $t \in (-t_0,t_0)$.

Let $t \in (-t_0, t_0)$, and let $(\Gamma_n)$ be a sequence in $\mathcal T_M$ 
that converges to $F_0$ in Hausdorff
metric. It is then an easy fact that the convergence
\begin{equation}\label{equacao_auxiliar_11}
\Gamma_n^t \to F^t_0  \qquad \text{ as } n \to \infty
\end{equation}
in the Hausdorff metric also holds true.

Take $R > 0$ such that $\supp h \subset D_R$. 
Since $h_t$ is a diffeomorphism that leaves $\C \setminus D_R$ invariant,  we then have
\begin{equation}\label{equality_perturbed_set}
	\Gamma_n^t \cap (\C \setminus D_R) = \Gamma_n \cap( \C \setminus D_R), 
\end{equation}
for all $n$. Since $\Gamma_n \in \mathcal T_M$, it does not intersect $\Delta_M$, and
so by \eqref{equality_perturbed_set}
\begin{equation} \label{Gamma_intersection_empty1} 
	\Gamma_n^t \cap (\C \setminus D_R) \cap \Delta_M  = \emptyset. 
	\end{equation}

The convergence \eqref{equacao_auxiliar_11}, together with the fact that $F_0^t$ does not intersect
the closed set $\overline{\Delta}_M$, see \eqref{equacao_auxiliar_11a},
implies that for large enough $n$,
\begin{equation} \label{Gamma_intersection_empty2}  
	\Gamma_n^t \cap D_R \cap \Delta_M  = \emptyset. 
	\end{equation}
Then by \eqref{Gamma_intersection_empty1} and \eqref{Gamma_intersection_empty2} we have
\[ \Gamma_n^t \cap \Delta_M = \emptyset \]
for $n$ large enough, and so $\Gamma_n^t \in \mathcal T_M$ by the definition 
\eqref{definition_subclass_cut_off} of the class $\mathcal T_M$.
Because of \eqref{equacao_auxiliar_11}, we then conclude that 
$F^t_0 \in \overline{\mathcal T}_M = \mathcal F_M$. 
\end{proof}

From  Lemma \ref{lemma_small_variations} it follows that the maximizer $F_0$
is a $\varphi$-critical set. Thus by Corollary \ref{cor_critical_set} its equilibrium
measure in the external field $\varphi$ is a $\varphi$-critical measure and the 
results of Section \ref{critical_measures}
apply to this measure. Thus we obtain the following proposition that summarizes the
work we did so far.

\begin{prop}\label{characterization_critical_measures_polynomial}
 The equilibrium measure $\mu_0=\mu^{\varphi,F_0}$ of $F_0$ is $\varphi$-critical. 
There exists a polynomial $R$ such that
\begin{equation}\label{quadratic_differential_polynomial_case}
 \left(C^{\mu_0}(z)+\frac{1}{2}V'(z)\right)^2=R(z), \quad m_2 \text{-a.e.}
\end{equation}
Moreover, $\supp\mu_0$ is compact and consists of a finite union of analytic arcs, 
which are maximal trajectories of the quadratic differential $-R(z)dz^2$, with each trajectory 
connecting two distinct zeros of $R$.
The measure $\mu_0$ is absolutely continuous with respect to the arclength measure on those arcs and
\begin{equation}\label{density_maximal_measure}
d\mu_0(s)=\frac{1}{\pi i}\sqrt{R(s)} \, ds
\end{equation}
where $ds$ is the complex line element. At any point $z\in\supp\mu_0$ which is not a zero of $R$, its logarithmic potential satisfies
\begin{equation}\label{equation_gradient_potential}
\frac{\partial }{\partial n_+}\left(U^{\mu_0}+\frac{1}{2}\varphi\right)(z)=\frac{\partial}{\partial
n_-}\left(U^{\mu_0}+\frac{1}{2} \varphi\right)(z),
\end{equation}
where $n_\pm$ are the unit normal vectors to $\supp\mu_0$ at $z$ pointing in opposite directions.
\end{prop}
Thus by \eqref{equation_gradient_potential} we have that the support of $\mu_0$ satisfies the S-property
in the external field $\varphi$.

\subsection{Proof of Theorem \ref{main_result}}  \label{proof_main_result}

Having in hands the equilibrium measure in the external field $\mu_0=\mu^{\varphi,F_0}$ of the maximizer set $F_0$ given 
by Corollary \ref{corollary_maximizing_set}, our final task is to construct a contour $\Gamma_0 \in \mathcal T$
for which $\mu^{\varphi,\Gamma_0}=\mu_0$. Note that we know that the support of $\mu_0$ consists of a finite
union of analytic arcs, but we do not know that the full set $F_0$  is a contour. There is actually
no reason why this should be the case, and so we will modify $F_0$ outside of the support of its
equilibrium measure in the external field, while preserving the equilibrium measure in the external field. 
 We also have to show that the contour $\Gamma_0$
that we obtain this way belongs to the class $\mathcal T$.

To this end, we consider the set 
\begin{equation}\label{variational_set}
\Lambda=\left\{z\in\C \mid U^{\mu_0}(z)+\frac{1}{2}\varphi(z)> l_0\right\},
\end{equation}
for the variational constant $l_0$ of $\mu_0$ appearing in the Euler-Lagrange variational
conditions \eqref{variational_condition_1} and \eqref{variational_condition_2} for $\mu_0$.

\begin{lem}\label{lemma_description_level_sets}\
%  \item[\it i)]The set $\overline\Lambda\cup\supp\mu_0$ has at most $N$ connected components, all of them unbounded.
  
For any $\epsilon > 0$, there exists $R_\epsilon>0$ such that for every $R \geq R_\epsilon$
the following holds. The set $\Lambda \setminus D_R$ has 
exactly $N$ connected components, each of them contained in precisely one of
the sectors 
\begin{equation}\label{disjoint_sets_infinity}
\left\{ z\in\C \mid |z|\geq R, |\arg z - \theta_j|\leq \frac{\pi}{2N}+\epsilon \right\},\quad j=1,\hdots,N,
\end{equation}
and containing in its interior the half-rays
\begin{equation}\label{definition_set_L_j}
L_{j,R} = \{z\in\C \mid |z|\geq R, \ \arg z=\theta_j\},\quad j=1,\hdots,N.
\end{equation}
\end{lem}
Recall that the angles $\theta_j$ are given in \eqref{definition_admissible_angles} and they are
determined by the polynomial $V$.

\begin{proof}
The boundary of $\Lambda$ consists of level curves of $U^{\mu_0} + \frac{1}{2} \varphi$, and these
are exactly the trajectories of the quadratic differential $-R(z)dz^2$ given in 
\eqref{quadratic_differential_polynomial_case}. The quadratic differential has a pole
of order $2N+2$ at $z=\infty$. Lemma \ref{lemma_description_level_sets} then follows 
from the local behavior of trajectories of a quadratic differential near a pole of order $\geq 4$, 
see \cite[Theorem~7.4]{strebel_book}.
\end{proof}

Fix a number
\begin{equation}\label{choice_epsilon}
0 < \epsilon<\frac{\pi}{2N}
\end{equation}
and also $R>R_{\epsilon}$ for which
\begin{equation}\label{choice_radius}
\supp \mu_0 \subset D_{R}.
\end{equation}
For $j=1, \hdots, N$, let $\Lambda(j)$ be the connected component of $\overline{\Lambda} \cup \supp \mu_0$ for which
\begin{equation}\label{definicao_lambda_j}
 L_{j,R} \subset \Lambda(j).
\end{equation}
Then $\Lambda(j)$ is pathwise connected, and we can connect any point in $\Lambda(j)$ to infinity 
through a curve entirely contained in 
$\Lambda(j)$ and stretching out to infinity in the sector $S_j$ (see \eqref{definition_admissible_angles}
for the definition of $S_j$).

The sets $\Lambda(j)$ for $j=1, \ldots, N$, need not be disjoint. 
It can happen that $\Lambda(j)=\Lambda(k)$ for some pair of distinct
numbers $j,k$. Of course, if their intersection is non-empty then $\Lambda(j) = \Lambda(k)$.

The next lemma is the key for completing the proof of Theorem \ref{main_result}. 

\begin{lem}\label{lemma_refinement_partition}
If $A \in \mathcal P_0$ and $j,k \in A$, then $\Lambda(j)=\Lambda(k)$.
\end{lem}

The proof of Lemma \ref{lemma_refinement_partition} is postponed to Section \ref{proof_lemma_refinement_partition}. 
We first show how to complete the proof of Theorem \ref{main_result}, assuming the lemma.

\begin{proof}[\it Proof of Theorem \ref{main_result}]
Let $\epsilon$ and $R > R_{\epsilon}$ be given as in \eqref{choice_epsilon} and \eqref{choice_radius}.
For the angles $\theta_1,\hdots,\theta_N$ in \eqref{definition_admissible_angles}, denote 
$$
z_j=Re^{i\theta_j},\quad j=1,\hdots, N.
$$

From \eqref{choice_radius} and Lemma \ref{lemma_description_level_sets} we see that $z_j\in \Lambda$.
Given $A \in \mathcal P_0$, we have by Lemma \ref{lemma_refinement_partition}
that the points $z_j$ with $j\in A$, belong to the
same connected component of $\overline\Lambda \cup \supp\mu_0$. Then
there exists a connected set $\gamma_A \subset
\overline\Lambda\cup\supp\mu_0$ which is a finite union of bounded $C^1$ Jordan arcs satisfying
$$
\quad z_j\in\gamma_A, \quad \text{for every } j\in A.
$$

For $L_{j,R}$ given by \eqref{definition_set_L_j}, define
$$
\Gamma_A=\gamma_A \cup \bigcup_{j\in A} L_{j,R}.
$$
Each $\Gamma_A$ is then a finite union of $C^1$ arcs. $\Gamma_A$ is connected and 
$\Gamma_A$ stretches out to infinity in each sector $S_j$ with $j\in A$. The union
$$
\Gamma_0=\bigcup_{A\in\mathcal P_0} \Gamma_A
$$
then satisfies all the requirements in Definition \ref{definition_admissible_sets}
and it follows that
\[ \Gamma_0 \in \mathcal T(\mathcal P). \]
Hence
\begin{equation}\label{equacao_auxiliar_18}
	I^\varphi(\Gamma_0)\leq \sup_{\Gamma\in \mathcal T(\mathcal P)} I^\varphi(\Gamma).
\end{equation}

Since $\Gamma_0\subset \Gamma_0\cup\supp\mu_0$, we also have that
\begin{equation}\label{equacao_auxiliar_19}
I^\varphi(\Gamma_0\cup\supp\mu_0)\leq I^\varphi(\Gamma_0).
\end{equation}

On the other hand, since $\Gamma_0 \subset \overline\Lambda\cup\supp\mu_0$,  
the variational conditions \eqref{variational_condition_1}, \eqref{variational_condition_2} 
are valid for $\Gamma_0\cup \supp\mu_0$, which means that $\mu_0$ is the
equilibrium measure of $\Gamma_0\cup \supp\mu_0$ in the external field $\varphi$. Thus 
\begin{equation} \label{equacao_auxiliar_19b}
I^\varphi(\Gamma_0\cup \supp\mu_0)  = I^\varphi(\mu_0) = I^\varphi(F_0), 
\end{equation}
since $\mu_0$ is the equilibrium measure in the external field of $F_0$ as well.
Then we recall that $F_0$ is the maximizer of $I^{\varphi}$ over the class $\mathcal F_M$ which
contains $\mathcal T_M$. Then by Proposition \ref{proposicao_fundamental}
\begin{equation}\label{equacao_auxiliar_19c}
	I^{\varphi}(F_0)  \geq \sup_{\Gamma\in \mathcal T_M}I^\varphi(\Gamma) = \sup_{\Gamma\in \mathcal T}I^\varphi(\Gamma)
\end{equation}
Combining the inequalities in \eqref{equacao_auxiliar_18}--\eqref{equacao_auxiliar_19c}, we finally get
\begin{equation} \label{equacao_auxiliar_19d}
I^\varphi(\Gamma_0) \leq \sup_{\Gamma \in \mathcal T} I^{\varphi}(\Gamma) 
	\leq I^{\varphi}(F_0) =  I^\varphi(\Gamma_0\cup \supp\mu_0) \leq I^{\varphi}(\Gamma_0)
\end{equation}
and so equality holds throughout in \eqref{equacao_auxiliar_19d}.

Let $\mu_1 = \mu^{\varphi,\Gamma_0}$ be the equilibrium measure in the external field $\varphi$ of $\Gamma_0$.
Then because of the equalities in \eqref{equacao_auxiliar_19d}, we have
\[ I^{\varphi}(\mu_1) = I^\varphi(\Gamma_0) = I^\varphi(\Gamma_0\cup \supp\mu_0) \]
Thus $\mu_1$ is a measure supported on $\Gamma_0 \subset \Gamma_0 \cup \supp\mu_0$ whose
energy in external field $\varphi$ is equal to the minimal energy for measures on $\Gamma_0 \cup \supp \mu_0$.
By uniqueness of equilibrium measure it follows that $\supp \mu_0 \subset \Gamma_0$
and $\mu_0 = \mu_1$ is the equilibrium measure in external field $\varphi$ of $\Gamma_0$.

Then $\Gamma_0 \in \mathcal T$ has the S-property in external field $\varphi$, see \eqref{equation_gradient_potential}
in Proposition \ref{characterization_critical_measures_polynomial}.
Also the remaining statements of Theorem \ref{main_result} follow from 
Proposition \ref{characterization_critical_measures_polynomial}.
This completes the proof of Theorem \ref{main_result}, assuming the validity of Lemma \ref{lemma_refinement_partition}.
\end{proof}

\subsection{Proof of Lemma \ref{lemma_refinement_partition}} \label{proof_lemma_refinement_partition}

What remains is the proof of Lemma \ref{lemma_refinement_partition}.

\begin{proof}[Proof of Lemma \ref{lemma_refinement_partition}]

We take $A \in \mathcal P_0$ and $j, k \in A$ with $j \neq k$ and our task
is to prove that $\Lambda(j) = \Lambda(k)$. 

In the proof we make use of the maximizer $F_0 \in \mathcal F_M$ and we know that
$\mu_0$ is the equilibrium measure of $F_0$ in the external field $\varphi$.
Thus by the variational conditions \eqref{variational_condition_1}, \eqref{variational_condition_2},
$$
F_0 \subset \{ z \in \mathbb C \mid U^{\mu_0}(z) + \frac{1}{2} \varphi(z) \geq l_0 \}. 
$$

If $z \in F_0 \setminus \supp \mu_0$ then $U^{\mu_0} + \frac{1}{2} \varphi$ is harmonic
near $z$, and so cannot have a local maximum at $z$.  It follows that 
\begin{equation} \label{equation_in_Lemma_1} 
	F_0 \subset \overline{\Lambda} \cup \supp \mu_0. 
	\end{equation}
where we recall that $\Lambda$ is defined in \eqref{variational_set}.

Since $\mathcal F_M$ is the closure of $\mathcal T_M$, there exists a sequence
$(\Gamma_n)_n$ in $\mathcal T_M$ such that
\begin{equation} \label{equation_in_Lemma_2} 
	\Gamma_n \to F_0 \quad \text{ in Hausdorff metric.} 
	\end{equation}

Let $n \in \mathbb N$. From item iii) of Definition \ref{definition_admissible_sets} it follows that
$\Gamma_n$ has a connected component that stretches out to infinity in the sectors
$S_j$ and $S_k$. Then  by dropping parts of that component that go to infinity
in other sectors, and parts that make loops, we can find a subset $\widetilde{\Gamma}_{n} \subset \Gamma_n$ 
that is a simple piecewise $C^1$ contour that goes from infinity
in the sector $S_j$ to infinity in the sector $S_k$. Here simple means that it has no points of self-intersection.
We consider $\widetilde{\Gamma}_n$ as an oriented contour, and we talk about
points on the contour that lie before or after other points.
From the definition of $\mathcal T_M$, see \eqref{definition_subclass_cut_off},
we know that 
\begin{equation} \label{equation_in_Lemma_3} 
	\widetilde{\Gamma}_n \cap \Delta_M  = \emptyset. 
	\end{equation}
The region $\Delta_M$ consists of $N$ connected components, all of them stretching out to infinity, 
as shown with the gray shaded regions in Figure \ref{level_curves} for the case $N = 5$.

%\begin{center}
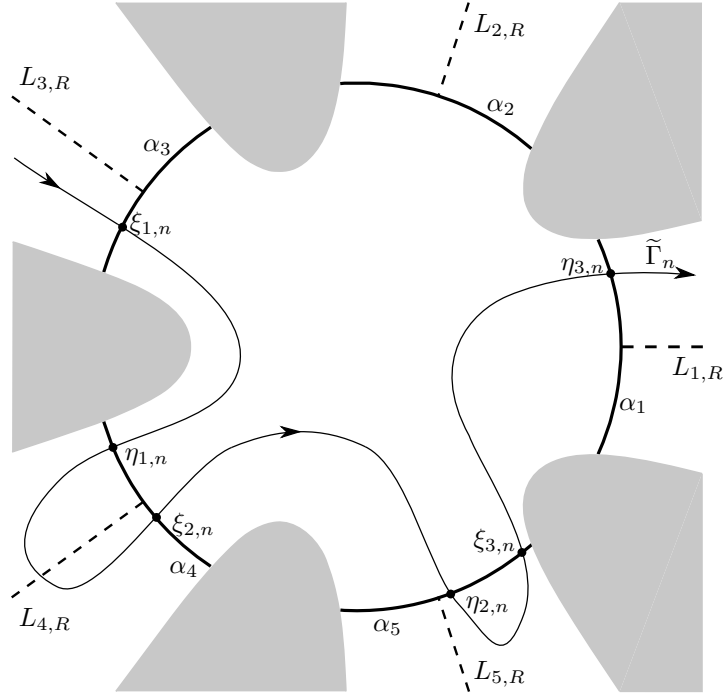
\begin{figure}
\begin{overpic}[scale=.6]{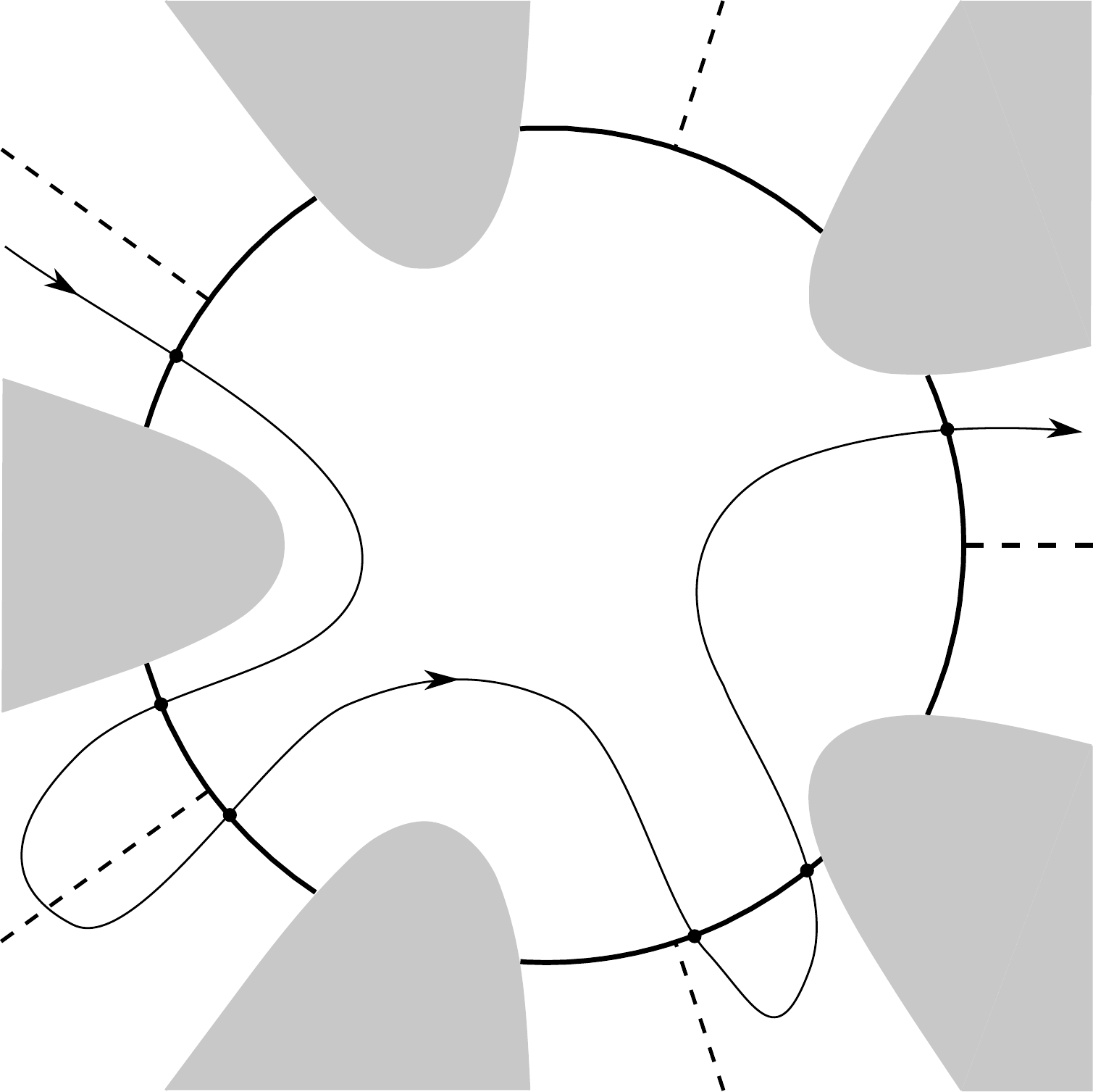}
  \put(250,120){$L_{1,R}$}
  \put(175,250){$L_{2,R}$}
  \put(03,230){$L_{3,R}$}
  \put(03,25){$L_{4,R}$}
  \put(175,05){$L_{5,R}$}
  \put(230,107){$\alpha_1$}
  \put(180,220){$\alpha_2$}
  \put(50,205){$\alpha_3$}
  \put(60,45){$\alpha_4$}
  \put(137,23){$\alpha_5$}
  \put(240,162){$\widetilde\Gamma_n$}
  \put(45,175){$\xi_{1,n}$}
  \put(43,88){$\eta_{1,n}$}
  \put(61,61){$\xi_{2,n}$}
  \put(172,32){$\eta_{2,n}$}
  \put(175,55){$\xi_{3,n}$}
  \put(208,160){$\eta_{3,n}$}
\end{overpic}
\caption{Pictorial example for the proof of Lemma \ref{lemma_refinement_partition}, considering 
a polynomial $V$ of degree $5$. The shaded region is $\Delta_M$. The black thick
arcs belonging to a circle are the curves $\alpha_i$. The oriented contour is an example of 
$\widetilde \Gamma_n$ for some $\Gamma_n$ in
\eqref{equation_in_Lemma_2}. The points $\xi_{i,n},\eta_{i,n}$, $j=1,\hdots,l$ are also 
represented in the figure, in this case with value $l=3$.}\label{level_curves_7}
\end{figure}
%\end{center}

Let $R$ be large enough such that the circle $|z| = R$ intersects each of the components
of $\Delta_M$ along a circular arc. Then there are also $N$ circular subarcs of $|z|= R$
outside of $\Delta_M$, that we call $\alpha_1, \ldots, \alpha_N$ where $\alpha_i$
is in the direction of the sector $S_i$, for $i = 1, \ldots, N$. By taking $R$ large enough,
we can also make sure that
\begin{equation} \label{equation_in_Lemma_4}
	\{ z \in \alpha_i  \mid U^{\mu_0}(z) + \frac{1}{2} \varphi(z) \geq l_0 \} \subset \Lambda(i)
	\end{equation}
for each $i$, where $\Lambda(i)$ is as in \eqref{definicao_lambda_j}.

We now follow the contour  $\widetilde{\Gamma}_n$ starting at infinity in the sector $S_j$.
It will meet the arc $\alpha_j$. We let $\xi_{1,n}$ be the \emph{last point} on $\widetilde{\Gamma}_n$
that is on $\alpha_j$. Then $\widetilde{\Gamma}_n$ enters into $D_R \setminus \Delta_M$
and never returns to $\alpha_j$.
Since it ends up at infinity again, the contour has to leave $D_R \setminus \Delta_M$ again, and it will
do so along  one of the arcs $\alpha_i$, with $i \neq j$. 
We let $\eta_{1,n}$ be the \emph{first point} after $\xi_{1,n}$ which is on one of the $\alpha_i$ again,
say $\eta_{1,n} \in \alpha_{j_1}$.
If $j_1 = k$, then we stop. If $j_1 \neq k$, then we continue and we let
$\xi_{2,n}$ be the \emph{last point} on $\widetilde{\Gamma}_n$ that is on $\alpha_{j_1}$.
After $\xi_{2,n}$ the contour is going into $D_R \setminus \Delta_M$ again, since the contour
has to end in the sector $S_k$ and $j_1 \neq k$. The contour will meet one of the $\alpha_i$ again,
and we let $\eta_{2,n}$ be the \emph{first point} after $\xi_{2,n}$ which is on one of the $\alpha_i$,
say $\eta_{2,n} \in \alpha_{j_2}$. Then $j_2 \neq j$ and $j_2 \neq j_1$, since $\xi_{1,n}$
is the last point on $\alpha_j$ and $\xi_{2,n}$ is the last point on $\alpha_{j_1}$, and $\eta_{2,n}$
comes after these two points. If $j_2 = k$ then we stop, and otherwise we continue, see 
Figure \ref{level_curves_7} for an illustration.

Continuing this process, we find a sequence of distinct indices 
$j_0, j_1, \ldots, j_l$ for some $l \in \{1, \ldots, N-1\}$,
where 
\[ j_0 = j, \qquad j_l = k, \] 
and points $\xi_{1,n}, \eta_{1,n}, \ldots, \xi_{l,n}, \eta_{l,n}$
on the contour $\widetilde{\Gamma}_n$ with
\begin{equation} \label{equation_in_Lemma_5} 
	\xi_{i,n} \in \alpha_{j_{i-1}}, \quad \text{ and } \quad \eta_{i,n} \in \alpha_{j_i} 
	\qquad \text{for } i = 1, \ldots, l. 
	\end{equation}
Furthermore, if we use $\widetilde{\Gamma}_n(\xi, \eta)$ to denote the 
part of $\widetilde{\Gamma}_n$ lying strictly between two points $\xi$ and $\eta$ of the contour,
then 
\begin{equation} \label{equation_in_Lemma_6} 
	\widetilde{\Gamma}_n(\xi_{i,n}, \eta_{i,n}) \subset D_R \setminus \Delta_M 
	\qquad \text{for } i = 1, \ldots, l. 
	\end{equation}

The numbers $l,j_1,\hdots,j_{l-1}$ depend on the contour $\widetilde{\Gamma}_n$ and so they depend on $n$. But these numbers are always positive integers smaller than or equal to
$N$, so there are just finite possible choices for them. By passing to a subsequence of the sequence $(\Gamma_n)$, we may assume they are independent of $n$.

The set
\[ \widetilde{\Gamma}_n[\xi_{i,n}, \eta_{i,n}] = \widetilde{\Gamma}_n(\xi_{i,n}, \eta_{i,n}) \cup \{\xi_{i,n}, \eta_{i,n}\} \]
is a connected closed subset of $ \widetilde{\Gamma}_n$ lying in the compact set $\overline{D_R} \setminus \Delta_M$.
By compactness of the Hausdorff metric, there is a subsequence, say $N_1 \subset \mathbb N$, such that
the limit
\begin{equation} \label{equation_in_Lemma_7} 
	F_i = \lim_{n \to \infty, n \in N_1} \widetilde{\Gamma}_n[\xi_{i,n}, \eta_{i,n}].  
	\end{equation}
exists in the Hausdorff metric for every $i = 1, \ldots, l$. 

The limit $F_i$ is connected, because this property is preserved by taking limits in the
Hausdorff metric if all the sets involved are contained in a large, but fixed, compact set of $\C$, in this case $\overline D_R$, see \eqref{equation_in_Lemma_6}.
Also
\begin{equation} \label{equation_in_Lemma_8} 
	F_i \subset F_0 \qquad \text{for } i = 1, \ldots, l, 
	\end{equation}
as follows from \eqref{equation_in_Lemma_2} and the fact $\widetilde{\Gamma}_n[\xi_{i,n}, \eta_{i,n}] \subset \Gamma_n$.

By compactness of $\alpha_i$ we may also assume (by taking a further subsequence, if necessary)
that the sequences $(\xi_{i,n})_n$ and $(\eta_{i,n})_n$ for $i = 1,\ldots, l$
converge along the same subsequence $N_1$, say
\[ \xi_i = \lim_{n \to \infty, n \in N_1} \xi_{i,n},
	\qquad \eta_{i} = \lim_{n \to \infty, n \in N_1} \eta_{i,n}, \quad \text{ for } i = 1, \ldots, l. \]
	
Then by \eqref{equation_in_Lemma_5} and \eqref{equation_in_Lemma_7} 
\begin{equation} \label{equation_in_Lemma_9} 
	\xi_i \in F_i \cap \alpha_{j_{i-1}}, \quad \text{ and } \quad \eta_i \in F_i \cap \alpha_{j_i}, 
	\end{equation}
for every $i = 1, \ldots, l$.
From \eqref{equation_in_Lemma_9} and \eqref{equation_in_Lemma_4} it follows that
\[ \xi_i \in \Lambda(j_{i-1}) \qquad \eta_i \in \Lambda(j_i). \]

Each $F_i$ is a connected subset of $F_0$ and by \eqref{equation_in_Lemma_1} and \eqref{equation_in_Lemma_9} 
we see that $\xi_i$ and $\eta_i$ belong to the same connected component of $\overline{\Lambda} \cap \supp \mu_0$,
which means that
\[ \Lambda(j_{i-1}) = \Lambda(j_i) \quad \text{ for  } i = 1, \ldots, l. \]

Thus $\Lambda(j_0) = \Lambda(j_l)$ which gives us $\Lambda(j) = \Lambda(k)$
as required.
\end{proof}

\bibliographystyle{amsplain}

\begin{thebibliography}{10}

\bibitem{alvarez_alonso_medina_s_curves}
G.~\'Alvarez, L.~Mart\'inez~Alonso, and E.~Medina, \emph{Determination
  of {S}-curves with applications to the theory of non-Hermitian orthogonal
  polynomials}, J. Stat. Mech. Theory Exp.,  \textbf{2013} (2013), no.~06, P06006.

\bibitem{ameur_hedenmalm_makarov}
Y.~Ameur, H.~Hedenmalm, and N.~Makarov, \emph{Random normal matrices and Ward identities}, 
preprint arXiv:1109.5941,  
Annals  Prob., to appear. 


\bibitem{aptekarev_arvesu_meixner}
A.~Aptekarev and J.~Arves\'u, \emph{Asymptotics for multiple Meixner polynomials},
J. Math. Anal. Appl. 411 (2014), 485--505. 

\bibitem{aptekarev_bleher_kuijlaars}
A.~I.~Aptekarev, P.~M.~Bleher, and A.~B.~J.~Kuijlaars, 
\emph{Large $n$ limit of Gaussian random matrices with external source II},
Comm. Math. Phys. 259 (2005), 367--389. 

\bibitem{aptekarev_kuijlaars_vanassche}
A.~I.~Aptekarev, A.~B.~J.~Kuijlaars, and W.~Van Assche,
\emph{Asymptotics of Hermite-Pad\'e rational approximants for two analytic functions with 
separated pairs of branch points (case of genus 0)},
Int. Math. Res. Pap. IMRP 2008, Art. ID rpm007, 128 pp. 

\bibitem{aptekarev_lysov_tulyakov}
A.~I.~Aptekarev, V.~G.~Lysov, and D.~N.~Tulyakov,
Random matrices with an external source and the asymptotics of  multiple orthogonal polynomials,
Mat. Sb. 202 (2011), no. 2, 3--56.
English translation in Sb. Math. 202 (2011), 155--206. 

\bibitem{atia_martinez_martinez_thabet}
M.J.~Atia, A.~Martínez-Finkelshtein, P.~Mart\'inez-González, and F.~Thabet,
\emph{Quadratic differentials and asymptotics of Laguerre
polynomials with varying complex parameters}, preprint
arXiv:1311:0372.

\bibitem{bergkvist_rullgard}
T.~Bergkvist and H.~Rullg{\aa}rd,  
\emph{On polynomial eigenfunctions for a class of differential operators},
Math. Res. Lett. \textbf{9} (2002), 153--171. 

\bibitem{bertola_boutroux}
M.~Bertola, \emph{Boutroux curves with external field: equilibrium measures
  without a variational problem}, Anal. Math. Phys. \textbf{1} (2011), 167--211.

\bibitem{bertola_mo_spinor}
M. Bertola and M. Y. Mo, \emph{Commuting difference operators, spinor bundles and the
              asymptotics of orthogonal polynomials with respect to varying
              complex weights}, Adv. Math. 220 (2009), 154--218.
   
\bibitem{bertola_tovbis_quartic_weight}
M.~Bertola and A.~Tovbis, \emph{Asymptotics of orthogonal polynomials
  with complex varying quartic weight: global structure, critical point
  behaviour and the first Painlev\'e equation}, preprint arXiv:1108.0321.

\bibitem{bleher_delvaux_kuijlaars} 
P.~Bleher, S.~Delvaux, and A.~B.~J.~Kuijlaars,
\emph{Random matrix model with external source and a constrained vector equilibrium problem},
Comm. Pure Appl. Math. 64 (2011),  116--160. 

\bibitem{cegrell_kolodziej_levenberg}
U.~Cegrell, S.~Kolodziej, and N.~Levenberg, \emph{Two problems on potential
  theory for unbounded sets}, Math. Scand. \textbf{83} (1998), no.~2, 265--276.

\bibitem{claeys_wielonsky}
T.~Claeys and F.~Wielonsky, \emph{On sequences of rational interpolants of the
  exponential function with unbounded interpolation points}, 
	J. Approx. Theory \textbf{171} (2013), 1--32.

\bibitem{deano_orthogonal_polynomials_bounded_interval}
A.~Deaño
\emph{Large degree asymptotics of orthogonal polynomials with respect to an oscillatory weight on a bounded interval}, preprint
arXiv:1402.2085.
	
\bibitem{deano_huybrechs_kuijlaars}
A.~Dea\~no, D.~Huybrechs, and A.~B.~J.~Kuijlaars, \emph{
Asymptotic zero distribution of complex orthogonal
polynomials associated with Gaussian quadrature},
J. Approx. Theory 162 (2010) 2202--2224.

	
\bibitem{deift_kriecherbauer_mclaughlin}
P.~Deift, T. Kriecherbauer, and K.T-R. McLaughlin,
\emph{New results on the equilibrium measure for logarithmic potentials in the presence of an 
external field}, J. Approx. Theory \textbf{95} (1998), 388--475. 

\bibitem{duits_painleve_kernels}
M.~Duits, \emph{Painlev\'e kernels in Hermitian matrix models},
  Constr. Approx. 39 (2014), 173--196. 
	
\bibitem{duits_kuijlaars_mo}
M.~Duits, A.~B.~J.~Kuijlaars, and M.~Y.~Mo,
\emph{The Hermitian two matrix model with an even quartic potential},
 Mem. Amer. Math. Soc. 217 (2012), no. 1022, v+105 pp. 


\bibitem{mclaughlin_ercolani_loop_equations}
N.~M.~Ercolani and K.~D.~T-R McLaughlin, \emph{A quick derivation of the loop
  equations for random matrices}, Probability, geometry and integrable systems,
  Math. Sci. Res. Inst. Publ., vol.~55, Cambridge Univ. Press, Cambridge, 2008,
  pp.~185--198.

\bibitem{gonchar_rakhmanov_rato_rational_approximation}
A.~A.~Gonchar and E.~A.~Rakhmanov, \emph{Equilibrium distributions and the rate
  of rational approximation of analytic functions}, Mat. Sb. (N.S.)
  \textbf{134(176)} (1987), no.~3, 306--352.  English
	 translation in Math. USSR-Sb. \textbf{62} (1989), no.~2, 305--348.

\bibitem{Greene_Krantz_book}
R.~E.~Greene and S.~G.~Krantz, \emph{Function theory of one complex
  variable}, third ed., Graduate Studies in Mathematics, vol.~40, American
  Mathematical Society, Providence, RI, 2006. 
  

\bibitem{shapiro_holst}  
T.~Holst and B.~Shapiro, 
\emph{On higher {H}eine-{S}tieltjes polynomials},
Israel J. Math. 183 (2011), 321--345.
  
  
\bibitem{huybrechs_lejon_kuijlaars}
D.~Huybrechs, A.~B.~J.~Kuijlaars, and N.~Lejon \emph{
Zero distribution of complex orthogonal polynomials with respect to exponential weights}, preprint
  arXiv:1312.4376.

\bibitem{kamvissis_rakhmanov_energy_maximization}
S.~Kamvissis and E.~A.~Rakhmanov, \emph{Existence and regularity
  for an energy maximization problem in two dimensions}, J. Math. Phys.
  \textbf{46} (2005),  083505, 24. 

\bibitem{kuijlaars_survey} 	
A.~B.~J.~Kuijlaars,
\emph{Multiple orthogonal polynomials in random matrix theory}, 
in: Proceedings of the International Congress of Mathematicians. Volume III,
(R.~Bhatia ed.), Hindustan Book Agency, New Delhi, India, 2010,
pp. 1417--1432.


\bibitem{Landkof_book}
N.~S. Landkof, \emph{Foundations of Modern Potential Theory}, Springer-Verlag,
  New York, 1972, Translated from the Russian by A. P. Doohovskoy, 
  Grundlehren der mathematischen Wissenschaften, Band 180.

\bibitem{martinez_orive_jacobi_single_contour}
A.~Mart{\'{\i}}nez-Finkelshtein and R.~Orive, \emph{Riemann-{H}ilbert analysis
  of {J}acobi polynomials orthogonal on a single contour}, J. Approx. Theory
  \textbf{134} (2005), no.~2, 137--170. 

\bibitem{martinez_rakhmanov}
A.~Mart{\'{\i}}nez-Finkelshtein and E.~A. Rakhmanov, \emph{Critical measures,
  quadratic differentials, and weak limits of zeros of {S}tieltjes
  polynomials}, Comm. Math. Phys. \textbf{302} (2011), no.~1, 53--111.
  

\bibitem{martinez_rakhmanov_suetin_variation_equilibrium_energy}
A.~Mart{\'{\i}}nez-Finkelshtein, E.~A. Rakhmanov, and S.~P. Suetin, \emph{Variation
  of equilibrium energy and the {S}-property of a stationary compact set},
  Mat. Sb. \textbf{202} (2011),  113--136. 
	English 	transl. in Sb. Math. \textbf{202} (2011), 1831--1852.
	

\bibitem{mattner_singular_integrals}
L.~Mattner, \emph{Strict definiteness of integrals via complete monotonicity of
  derivatives}, Trans. Amer. Math. Soc. \textbf{349} (1997), 3321--3342.
	
	
\bibitem{nuttall}
J.~Nuttall, \emph{Hermite–Pad\'e approximants to functions meromorphic on a Riemann surface},
J. Approx. Theory 32 (1981), 233--240.

\bibitem{rakhmanov_orthogonal_s_curves}
E.~A. Rakhmanov, \emph{Orthogonal polynomials and {S}-curves}, Contemp.
  Math., vol. 578, Amer. Math. Soc., Providence, RI, 2012. 
	
\bibitem{rakhmanov_suetin}
E.~A. Rakhmanov and S.~P. Suetin,
\emph{The distribution of the zeros of the Hermite-Pad\'e polynomials for a pair
of  functions forming a Nikishin system},
Mat. Sb. \textbf{204} (2013), no.~9, 115--160.
English transl. in Sb. Math. \textbf{204} (2013), 1347--1390.

\bibitem{ransford_book}
T.~Ransford, \emph{Potential Theory in the Complex Plane}, London
  Mathematical Society Student Texts, vol.~28, Cambridge University Press,
  Cambridge, 1995.

\bibitem{Saff_book}
E.~B.~Saff and V.~Totik, \emph{Logarithmic Potentials with External
  Fields}, Grundlehren der Mathematischen Wissenschaften, vol. 316, Springer-Verlag, Berlin,
  1997.

\bibitem{shapiro_takemura_tater}  
B.~Shapiro, K.~Takemura, and M.~Tater, 
\emph{On spectral polynomials of the {H}eun equation. {II}},
Comm. Math. Phys. 311 (2013), 277--300.
  
\bibitem{shapiro_tater}  
B.~Shapiro and M.~Tater,
\emph{On spectral polynomials of the {H}eun equation. {I}},
J. Approx. Theory 162 (2010), no.~4, 766--781.   
  
\bibitem{stahl_extremal_domains}
H.~Stahl, \emph{Extremal domains associated with an analytic function.
  {I}, {II}}, Complex Variables Theory Appl. \textbf{4} (1985), no.~4,
  311--324, 325--338. 

\bibitem{stahl_structure_extremal_domains}
\bysame, \emph{The structure of extremal domains associated with an analytic
  function}, Complex Variables Theory Appl. \textbf{4} (1985), 339--354.

\bibitem{stahl_orthogonal_polynomials_complex_weight_function}
\bysame, \emph{Orthogonal polynomials with complex-valued weight function. {I},
  {II}}, Constr. Approx. \textbf{2} (1986), 225--240, 241--251.

\bibitem{stahl_orthogonal_polynomials_complex_measures}
\bysame, \emph{Orthogonal polynomials with respect to complex-valued measures},
  Orthogonal polynomials and their applications ({E}rice, 1990), IMACS Ann.
  Comput. Appl. Math., vol.~9, Baltzer, Basel, 1991, pp.~139--154.

\bibitem{stahl_sets_minimal_capacity}
\bysame, \emph{Sets of minimal capacity and extremal domains}, preprint
  arXiv:1205.3811v1.

\bibitem{strebel_book}
K.~Strebel, \emph{Quadratic Differentials}, Ergebnisse der Mathematik und
  ihrer Grenzgebiete, 
  vol.~5, Springer-Verlag, Berlin, 1984.

\end{thebibliography}

\end{document}